\documentclass[journal]{IEEEtran}
\usepackage{diagbox}
\usepackage{amssymb}
\usepackage{amsmath}
\usepackage{array}
\usepackage{graphicx}
\usepackage{subcaption}
\usepackage{subcaption,siunitx,booktabs}
\usepackage{caption}
\usepackage{lipsum}
\usepackage{algorithmic}
\usepackage{algorithm}
\usepackage{amsthm}
\usepackage{mathrsfs}
\usepackage{enumitem}
\usepackage[colorlinks=true,allcolors=blue]{hyperref}
\newtheoremstyle{customtheorem}
  {1pt}
  {1pt}
  {}
  {10pt}
  {\itshape}
  {:}
  {.5em}
  {\thmname{#1} \thmnumber{#2} \thmnote{(#3)}} 
\theoremstyle{customtheorem}
\newtheorem{theorem}{Theorem}
\newtheoremstyle{customlemma}
  {1pt}
  {1pt}
  {}
  {10pt}
  {\itshape}
  {:}
  {.5em}
  {\thmname{#1} \thmnumber{#2} \thmnote{(#3)}} 
\theoremstyle{customlemma}
\newtheorem{lemma}{Lemma}

\newtheoremstyle{customdefinition}
  {1pt}
  {1pt}
  {}
  {10pt}
  {\itshape}
  {:}
  {.5em}
  {\thmname{#1} \thmnumber{#2} \thmnote{(#3)}} 
\theoremstyle{customdefinition}
\newtheorem{definition}{Definition}

\newtheoremstyle{customproposition}
  {1pt}
  {1pt}
  {}
  {10pt}
  {\itshape}
  {:}
  {.5em}
  {\thmname{#1} \thmnumber{#2} \thmnote{#3}}
  \theoremstyle{customproposition}
\newtheorem{proposition}{Proposition}
\newtheoremstyle{customremark}
  {1pt}
  {1pt}
  {}
  {10pt}
  {\itshape}
  {:}
  {.5em}
  {\thmname{#1} \thmnumber{#2} \thmnote{(#3)}}
\theoremstyle{customremark}

\makeatletter
\renewenvironment{proof}[1][\proofname]{\par
  \pushQED{\qed}%
  \normalfont \topsep=2pt \partopsep=2pt 
  \trivlist
  \item[\hspace{15pt}\itshape #1\@addpunct{:}]\ignorespaces
}{%
  \popQED\endtrivlist\@endpefalse
}

\usepackage{bm}
\usepackage[style=ieee]{biblatex}

\begin{document}

\title{Proximal Dogleg Opportunistic Majorization for Nonconvex and Nonsmooth Optimization}

\author{Yiming~Zhou and Wei~Dai\thanks{The authors are with the Department of Electrical and Electronic Engineering, Imperial College London, SW7 2AZ London, U.K. (yiming.zhou18@imperial.ac.uk; wei.dai1@imperial.ac.uk)}}
\markboth{IEEE }%
{Shell \MakeLowercase{\textit{et al.}}: Bare Demo of IEEEtran.cls for IEEE Journals}

\maketitle
\begin{abstract}

We consider minimizing a function consisting of a quadratic term and a proximable term which is possibly nonconvex and nonsmooth. This problem is also known as scaled proximal operator. Despite its simple form, existing methods suffer from slow convergence or high implementation complexity or both. To overcome these limitations, we develop a fast and user-friendly second-order proximal algorithm. Key innovation involves building and solving a series of opportunistically majorized problems along a hybrid Newton direction. The approach directly uses the precise Hessian of the quadratic term, and calculates the inverse only once, eliminating the iterative numerical approximation of the Hessian, a common practice in quasi-Newton methods. The algorithm's convergence to a critical point is established, and local convergence rate is derived based on the Kurdyka-Łojasiewicz property of the objective function. Numerical comparisons are conducted on well-known optimization problems. The results demonstrate that the proposed algorithm not only achieves a faster convergence but also tends to converge to a better local optimum compare to benchmark algorithms.


\end{abstract}


\IEEEpeerreviewmaketitle

\section{Introduction}
\IEEEPARstart{M}{any} signal processing and machine learning problems can be formulated as the composition of a quadratic (smooth) term and a regularization term which can be nonsmooth and nonconvex. That is
\begin{equation}
\label{eq:composite-problem}
    \min_{\bm{x} \in \mathbb{R}^{n}} f(\bm{x}) := \underbrace{\frac{1}{2} \bm{x}^{T} \bm{Q} \bm{x}+\bm{b}^{T} \bm{x} }_{q(\bm{x})} + h(\bm{x}),
\end{equation}
where $\bm{x}$ is the decision variable, and $\bm{Q} \succ 0$
\footnote{When $\bm{Q}$ is rank deficient, an $\epsilon \bm{I}$ can be added into $\bm{Q}$ where $\epsilon > 0$ is small.} 
and $\bm{b}$ are given parameters. 
We assume that  $h$ is proximable, i.e., its proximal operator \cite{parikh2014proximal}
\begin{equation}
\operatorname{prox}_{\tau h}(\bm{y}):=\arg \min_{\bm{x} \in \mathbb{R}^n} \left\{h(\bm{x})+\frac{1}{2 \tau}\|\bm{x}-\bm{y}\|^2\right\}
\label{eq:prox-operator}
\end{equation}
with $\tau > 0$ is easy to compute. 
We further assume that the overall objective function $f$ is lower bounded. 

This formulation \eqref{eq:composite-problem} finds many applications in signal processing and machine learning tasks.  A well-known example is compressed sensing (CS) \cite{donoho2006compressed}, which has been applied to medical imaging \cite{trzasko2008highly}, vibration monitoring \cite{wang2018nonconvex}, sparse robust signal recovery \cite{suzuki2023sparse}, etc. CS recovery can be formulated in \eqref{eq:composite-problem} where the quadratic term $q(\bm{x})$ enforces data fidelity and the regularization term $h(\bm{x})$ promotes sparsity. Commonly used regularizers include $\ell_1$ norm \cite{candes2008introduction,chen2001atomic,davenport2012introduction}, $\ell_0$ pseudo-norm \cite{blumensath2009iterative}, minimax concave penalty \cite{zhang2010nearly}, and capped-$\ell_1$ penalty \cite{zhang2010analysis}. 
Other examples include robust low-rank matrix completion \cite{chen2015fast,guo2018low,huang2021robust}, robust principle component analysis (RPCA)\cite{candes2011robust,chandrasekaran2011rank,ma2018efficient} and robust recovery of subspace structures \cite{liu2012robust} for machine learning. The regularization term is designed to explore the underlying low-rank structure of the solution, e.g., nuclear norm or the indicator function of the matrix rank. See Section \ref{sec:simulations} for examples tested in this paper. 

\subsection{Brief Discussions of the Literature}

A natural choice for solving \eqref{eq:composite-problem} is the proximal gradient (PG) method, also known as the forward-backward splitting (FBS) \cite{combettes2011proximal,parikh2014proximal,antonello2018proximal}.  The PG generalizes classical gradient descent from the smooth case to the nonsmooth case by introducing proximal operators. Its convergence to a critical point has been intensively studied in the literature --- for convex $f$ \cite{sahu2021convergence} and for nonconvex $f$ \cite{attouch2009convergence,attouch2013convergence,jia2023convergence} --- based on the Kurdyka-Łojasiewicz (KL) property. The established results indicate that the convergence rate is not superior to a sublinear rate of order $O(1 / k)$, where $k$ is the iteration count.

Nesterov extrapolation \cite{nesterov1983method} is an effective acceleration scheme to improve the convergence rate. 
With convex $f$, the convergence rate is accelerated to $O\left(1 / k^2\right)$ \cite{beck2009fast}. For nonconvex $f$, the accelerated PG (APG) algorithm \cite{li2015accelerated} chooses between the standard proximal gradient step and the accelerated step in each iteration, and accepts the one leading to lower value in the objective function. However, linear or sublinear convergence rate of APG is proved only under certain conditions \cite{li2015accelerated,gu2018inexact,attouch2009convergence,attouch2013convergence,jia2023convergence}.

Recent research has focused on Newton-type algorithms. The proximal Newton method \cite{lee2014proximal} considers an objective function comprising a proximable term and a second-order differentiable term. This iterative algorithm constructs a scaled proximal operator \eqref{eq:composite-problem} in each iteration, derived from the Hessian of the differentiable term, and then solves it. Assuming both strong convexity of the objective function and efficient solvability of the scaled proximal operator \eqref{eq:composite-problem} in each iteration, the method achieves a superlinear asymptotic convergence rate. 
Subsequent work in \cite{adler2020new} reduces computational efforts by adopting Shamanskii’s philosophy \cite{Shamanskii1967AMO}, updating the Hessian once in every $n$ iterations, where the value of $n$ influences the practical convergence behavior of the algorithm. However, solving the scaled proximal operator \eqref{eq:composite-problem} poses a computational challenge and devising an efficient solver remains an open problem in current research. Furthermore, the observed and theoretically proven fast convergence in above works is limited to convex cases, with no guarantee for nonconvex problems.

The quasi-Newton approach in \cite{patrinos2013proximal,stella2017forward,stella2017simple,themelis2018forward,de2022proximal} avoids direct manipulation of the scaled proximal operator. It introduces a forward-backward envelope (FBE) of the objective function, ensuring equivalence in minimization results with that of the original function. During iterations, the envelope's gradient is computed, a Hessian approximation for the overall envelope is obtained using quasi-Newton mechanisms like BFGS and L-BFGS \cite{liu1989limited}, and a line search is performed along the obtained quasi-Newton direction. However, the fast convergence requires the isolatedness of the limit point (the strong local convexity).\footnote{
Although not explicitly stated in \cite[Theorem 2.6]{stella2017forward}, the gradient and the Hessian of the envelope are not well defined at the points where the solution of the proximal operator is not unique. 
} Additionally, implementing quasi-Newton approaches for large-scale problems is not always convenient.

\subsection{Contributions}

In this paper, we develop a second-order proximal algorithm, named the Proximal Dogleg Opportunistic Majorization (PDOM) algorithm, for the nonconvex and nonsmooth problem \eqref{eq:composite-problem}. Our contributions are summarized as follows.
\begin{enumerate}

    \item The central concept PDOM is majorization-minimization (MM) wherein the surrogate function is crafted along a dogleg path which combines both gradient and Newton directions. The gradient direction guarantees the correctness of the critical point of the algorithm; and the Newton direction allows fast convergence.

    \item The computational cost of PDOM (per iteration) is lower compared to FBE-based quasi-Newton methods. Like all Newton-type algorithms, a backtracking is necessary for PDOM. However, unlike FBE-based methods where the Hessian approximation and Newton direction are updated at each iteration, PDOM calculates both the Hessian inverse and the Newton direction only once during algorithm initialization, thereby reducing computational cost.

    \item We prove that PDOM converges to a critical point. Moreover, local convergence rates of PDOM are derived for all three different regimes of the Łojasiewicz exponent, assuming the Kurdyka-Łojasiewicz property.

    
    \item Numerical evaluations have been conducted on well-known nonconvex problems. Specifically, the evaluations have shown the rapid convergence of PDOM and empirically revealed that PDOM finds better local minimizers compared to benchmark algorithms.

\end{enumerate}

\subsection{Notations}
Throughout this paper, we use $\mathbb{R}^{n}$ to denote the $n$-dimensional Euclidean space. The symbols $\langle\cdot, \cdot\rangle$ and $\|\cdot\|$ denote the standard inner product and norm in the space $\mathbb{R}^n$. For any $\bm{x} \in \mathbb{R}^{n}$, the $\ell_2$ norm, the $\ell_1$ norm, and the $\ell_0$ norm pseudo-norm are defined by $\|\bm{x}\|_2:=\sqrt{\bm{x}^{T} \bm{x}}$, $\|\bm{x}\|_1:=\sum_{i=1}^n\left|x_i\right|$, and $\|\bm{x}\|_0:=\left|\operatorname{supp}(\bm{x})\right|$ where $\operatorname{supp}(\cdot)$ counts the number of nonzero elements in $\bm{x}$. Given a positive semidefinite matrix $\bm{Q} \in \mathbb{R}^{n \times n}$, the scaled norm of $\bm{x}$ is defined as $\|\bm{x}\|_{\bm{Q}}:=\sqrt{\bm{x}^{\mathrm{T}} \bm{Q} \bm{x}}$. Given a closed set $\Omega \subseteq \mathbb{R}^n$, $\operatorname{dist}\left(\bm{x},\Omega\right) :=\inf\left\{\|\bm{y}-\bm{x}\|_2: \bm{y} \in \Omega\right\}$ calculates the distance between $\bm{x}$ and $\Omega$.


\section{PRELIMINARIES}

\begin{definition}
A function $f: \mathbb{R}^n \rightarrow(-\infty,+\infty]$ is said to be proper if  $\operatorname{dom} f \neq \emptyset$, where $\operatorname{dom} f=\{\bm{x} \in \mathbb{R}^n: f(\bm{x})<+\infty\}$, and lower semicontinuous at point $\bm{x}_{0}$ if
\begin{equation}
\label{eq:lsc-definition}
\liminf _{\bm{x} \rightarrow \bm{x}_0} f(\bm{x}) \geq f\left(\bm{x}_0\right).
\end{equation}
\end{definition}

\begin{definition}
    A function $f: \mathbb{R}^n \rightarrow \mathbb{R}$ is said to has a Lipschitz gradient if for all $\bm{x},\bm{y} \in \operatorname{dom}f$ it holds that
    \begin{equation}
    \|\nabla f(\bm{x})-\nabla f(\bm{y})\| \leq L\|\bm{x}-\bm{y}\|.
    \end{equation}
    The Lipschitz constant of the gradient, denoted as $L_{f}$, is defined as the smallest value that satisfies this inequality. 
\end{definition}
The value of $L_{f}$ for a twice differentiable function $f: \mathbb{R}^n \rightarrow \mathbb{R}$ with a positive
    semi-definite Hessian matrix $\bm{Q} \in \mathbb{R}^{n \times n}$ is the largest eigenvalue of $\bm{Q}$ denoted $\lambda_{\max}(\bm{Q})$.

\begin{definition}[Subdifferential \cite{rockafellar2009variational}] Let $h: \mathbb{R}^n \rightarrow \mathbb{R} \cup\{+\infty\}$ be a proper and lower semicontinuous function. For a given $\bm{x} \in$ dom $h$, the Frechet subdifferential of $h$ at $\bm{x}$, written as $\hat{\partial} h(\bm{x})$, is the set of all vectors $\bm{v} \in \mathbb{R}^n$ which satisfy
$$
\liminf _{\bm{y} \neq \bm{x}, \bm{y} \rightarrow \bm{x}} \frac{h(\bm{y})-h(\bm{x})-\langle\bm{v}, \bm{y}-\bm{x}\rangle}{\|\bm{y}-\bm{x}\|} \geq 0
$$
The subdifferential (which is also called the limiting subdifferential) of $h$ at $\bm{x}\in \operatorname{dom}h$, written as $\partial h(\bm{x})$, is defined by
    \begin{align}
    \partial h(\bm{x}):=\{\bm{v} \in \mathbb{R}^n : & \exists \bm{x}^k \rightarrow  \bm{x}, h\left(\bm{x}^k\right) \rightarrow h(\bm{x}),
    \nonumber
    \\
    &
    \bm{v}^k \in \hat{\partial} h\left(\bm{x}^k\right) \rightarrow \bm{v}, k \rightarrow \infty\}.
    \label{eq:sub-definition}
    \end{align}
    A point $\bm{x}^{\star} \in \operatorname{dom} \partial h$ is called a critical point of $h$ if $\bm{0} \in \partial h(\bm{x}^{\star})$, in which we define $\operatorname{dom} \partial h:=\{\bm{x} \in \mathbb{R}^n: \partial h(\bm{x}) \neq \emptyset\}$.
\end{definition}

\begin{definition}[Kurdyka-Łojasiewicz property \cite{attouch2010proximal}]
    A proper closed function $h: \mathbb{R}^n \rightarrow \mathbb{R} \cup\{+\infty\}$ is said to have the Kurdyka-Łojasiewicz (KL) property at $\hat{\bm{x}} \in \text{dom}\partial h$ if there exists $\eta \in (0,+\infty]$, a neighborhood $\mathcal{B}_\rho(\hat{\bm{x}}) \triangleq\{\bm{x}:\|\bm{x}-\hat{\bm{x}}\|<\rho\}$,  and a continuous desingularizing concave function $\psi:[0, \eta) \rightarrow[0,+\infty)$ with $\psi(0)=0$ such that,
    \begin{enumerate}[label=(\roman*)]
    \item $\psi$ is a continuously differentiable function with $\psi^{\prime}(x)>0$, $\forall x \in(0, \eta)$,
    \item for all $\bm{x} \in \mathcal{B}_\rho(\hat{\bm{x}}) \cap$$\{\bm{u} \in \mathbb{R}^n: h(\hat{\bm{x}})<h(\bm{x})<h(\hat{\bm{x}})+\eta\}$, it holds that
    \begin{equation}
    \label{eq:KL-inequality}
    \psi^{\prime}(h(\bm{x})-h(\hat{\bm{x}})) \operatorname{dist}(0, \partial h(\bm{x}))>1.
    \end{equation}
    \end{enumerate}
A proper closed function $h$ satisfying the KL property at all points in $\operatorname{dom} \partial h$ is called a KL function.
\end{definition}
\begin{definition}[Łojasiewicz exponent \cite{yu2022kurdyka}]
For a proper closed function $h$ satisfying the KL property at $\hat{\bm{x}} \in \operatorname{dom}\partial h$, if the desingularizing function $\psi$ can be chosen as $\psi(t) = \frac{C}{1-\theta} t^{1-\theta}$ for some $C > 0$ and $\theta \in [0,1)$, i.e., there exist $\rho > 0$ and $\eta \in (0,+\infty]$ so that
\begin{equation}
\label{eq:KL-inequality-1}
\operatorname{dist}(0, \partial h(\bm{x})) \geq C(h(\bm{x})-h(\hat{\bm{x}}))^{\theta},
\end{equation}
where $\bm{x} \in \mathcal{B}_\rho(\hat{\bm{x}})$ and $h(\hat{\bm{x}})<h(\bm{x})<h(\hat{\bm{x}})+\eta$, then we say that $h$ has the KL property at $\hat{\bm{x}}$ with an exponent of $\theta$. We say that $h$ is a KL function with an exponent of $\theta$ if $h$ has the same exponent $\theta$ at any $\hat{\bm{x}} \in \operatorname{dom} \partial h$.
\end{definition}

A wide range of functions encountered in optimization problems have the KL property. An example is that all proper closed semi-algebraic or  subanalytic functions are KL functions with the exponent $\theta \in [0,1)$ \cite{attouch2010proximal}. The value of Łojasiewicz exponent determines the local convergence rate of the PDOM. In Subsection \ref{subsec:rate-analysis} , we provide the exponent value of the problem under test. To the best of our knowledge, the Łojasiewicz exponent for the RPCA formulation \eqref{eq:RPCA} is first established in this literature. 

\subsection{PG from the Majorization-minimization Angle}
\label{subsec:PG-MM}
In this subsection, we analyze the PG as a majorization-minimization algorithm and point out limitations causing the slow convergence rate. The PG algorithm is a classical method to solve composite optimization problems \eqref{eq:composite-problem}. Each iteration of the PG can be viewed as a ''proximal line search'' conducted along the (negative) gradient direction with a positive step size $\tau$. At a given point $\bm{x}^k$, PG solves the surrogate function
\begin{align}
&\underbrace{q\left(\bm{x}^k\right)+\left\langle\nabla q\left(\bm{x}^k\right), \bm{x}-\bm{x}^k\right\rangle+\frac{1}{2 \tau}\left\|\bm{x}-\bm{x}^k\right\|^2}_{m_{g}(\bm{x}; \bm{x}^k)} + h(\bm{x})\nonumber \\
&= \frac{1}{2 \tau}\left\|\bm{x}-\left(\bm{x}^k-\tau \nabla q\left(\bm{x}^k\right)\right)\right\|^2+h(\bm{x})+c,\label{eq:pg-subproblem}    
\end{align}
where $c\in \mathbb{R}$ is a constant. The solution of \eqref{eq:pg-subproblem} is denoted as the (unscaled) proximal operator:
\begin{equation}
    \label{eq:standard-prox}
    \begin{aligned}
    \bm{x}^{k+1}  &=\operatorname{prox}_{\tau h}\left(\bm{x}^k-\tau \nabla q\left(\bm{x}^k\right)\right) \\
    & :=\arg \min _{\bm{x}} h(\bm{x})+\frac{1}{2 \tau}\left\|\bm{x}-\left(\bm{x}^k-\tau \nabla q\left(\bm{x}^k\right)\right)\right\|^2.
    \end{aligned}
\end{equation}
Under the assumption that $h(\cdot)$ is proximable, solving \eqref{eq:standard-prox} is computationally straightforward. 

The generated sequence leads to a non-increasing objective value, where the surrogate function $m_{g}(\bm{x}; \bm{x}^k)$ majorizes $q(\bm{x})$, i.e., $m_{g}(\bm{x}; \bm{x}^k) \geq q(\bm{x})$ for all $\bm{x} \in \operatorname{dom}f$ with $\tau < 1/L_{q}$. 
\begin{align}
    f(\bm{x}^{k+1}) &= q(\bm{x}^{k+1}) + h(\bm{x}^{k+1})\stackrel{(a)}{\leq} m_{g}(\bm{x}^{k+1})+h(\bm{x}^{k+1})\nonumber\\
    &\stackrel{(b)}{\leq} m_{g}(\bm{x}^{k})+h(\bm{x}^{k}) = f(\bm{x}^{k}),
    \label{eq:pg-monotone}
\end{align}
where $(a)$ is due to the majorization step, and $(b)$ is a consequence of the proximal operator. Provided that $f$ is lower bounded, the convergence of $f\left( \bm{x}^k \right)$ is guaranteed. 

However, the gradient direction may lead to a slow convergence, especially for solving nonconvex functions \cite{gu2018inexact,attouch2009convergence,attouch2013convergence,jia2023convergence}. Meanwhile, as observed in \cite[Chapter 9]{boyd2004convex}, when the Hessian of a function has a large condition number, typically on the order of 1000 or greater, the gradient-based method exhibits prohibitively slow convergence, rendering it impractical for real-world applications.

\section{The PDOM algorithm and convergence analysis}

In this section, we introduce the PDOM algorithm to solve the possible nonconvex and nonsmooth problem \eqref{eq:composite-problem}, addressing the limitations outlined in Subsection \ref{subsec:PG-MM}. We adopt the dogleg search, originally from the trust region method, in the PDOM algorithm by replacing the descent direction from the gradient to a hybrid direction. Specifically, the new direction is a convex combination of the gradient and Newton's directions of $q$. Different from the trust region method which minimizes the objective function along the hybrid direction within the trust region, PDOM solves a series of majorized problems. In particular, the surrogate is only require to majorize the objective function along the line connecting the current iterate and the path point which is a weaker condition compared to global upper bound.

\subsection{Hybrid direction and opportunistic majorization}
Given $\alpha \in (0,2]$, the dogleg path is denoted as \footnote{The path can be written more compactly as a single expression, but for the convenience of subsequent discussions, we use the form \eqref{eq:dogleg-standard}.}
\begin{align}
    \bm{p}(\alpha)
    & 
    := \begin{cases}
         \bm{p}_{\tau}
        & \alpha \in \left( 0,1 \right], \\
        \bm{p}_{\tau} 
        + (\alpha - 1) \left( \bm{p}_N - \bm{p}_{\tau} \right) 
        & \alpha \in \left( 1,2 \right],
    \end{cases}
    \label{eq:dogleg-standard}
\end{align}
where
\begin{align}
    \bm{g}
    := \nabla q(\bm{x}) = \bm{Q}\bm{x} + \bm{b},\quad
    \bm{p}_{\tau}
    := - \tau \bm{g},\quad
    \bm{p}_N 
    := - \bm{Q}^{-1} \bm{g},
    \nonumber
\end{align}
and $\tau$ is the fixed step size of the gradient direction within $(0,1/L_{q})$, and $\bm{p}_{N}$ denotes Newton point. The gradient direction is essential for ensuring convergence to a critical point, because at $\bm{x}^{\text{cri}}$, the first-order optimality condition of \eqref{eq:composite-problem} implies $\bm{0} \in \partial f(\bm{x}^{\text{cri}}) = \nabla q(\bm{x}^{\text{cri}}) + \partial h(\bm{x}^{\text{cri}})$, where the gradient direction is needed.

The path differs from the one in \cite[Chapter 4]{NoceWrig06}, with the scaling factor $\alpha$ excluded from the first segment. This difference arises because $\bm{p}_{\tau}$ consistently functions as the descent direction, and the trust-region radius constraint is not considered. Despite this distinction, our path remains a descent direction for the quadratic term.
\begin{lemma}
    \label{lem:zero-inner-product}
    The equality in the following equation,
    \begin{equation}
    \label{eq:q-p-inner-product}
    \langle \bm{p}(\alpha),\nabla q(\bm{p}(\alpha)) \rangle \le 0,
    \end{equation}
    is satisfied when $\tau = 1/\lambda_{\max}$, where $\lambda_{\max}$ denotes the largest eigenvalue of $\bm{Q}$. For any other $\tau \in (0,1/\lambda_{\max})$, the strict inequality holds.
\end{lemma}
\begin{proof}
    See Appendix \ref{Appendix-B}.
\end{proof}

The key to the success of an MM algorithm lies in constructing a surrogate function, serving as an upper bound of the objective function. In PDOM, the local surrogate function $m_{\alpha}$ is the projection of $m_{g}$ onto the path direction, that is 
\begin{align}
    &m_{\alpha}(\bm{x}; \bm{x}^k) 
    := q(\bm{x}^k)+
    \left\langle \bm{g}_{\alpha},\bm{x}-\bm{x}^k \right\rangle 
    + \frac{ 1 }{2 \tau_{\alpha}} 
    \left\| \bm{x} -\bm{x}^k\right\|^2
    \nonumber\\
    & 
    = q(\bm{x}^k)+\frac{ 1 }{2 \tau_{\alpha}} 
    \left\| \bm{x} -\left(\bm{x}^k+\bm{p}(\alpha)\right) \right\|^2
    - \frac{\tau_{\alpha}}{2}
    \left\| \bm{g}_{\alpha} \right\|^2,
    \label{eq:m-unscaled}
\end{align}
where, for $\alpha \in \left(0, 2 \right]$,
\begin{align}
    \bm{g}_{\alpha}
    :=  
    \frac{\langle \bm{g}, \bm{p}(\alpha) \rangle}{\|\bm{p}(\alpha)\|^2} \bm{p}(\alpha), \quad
    \tau_{\alpha}
    := - 
    \frac{ \left\| \bm{p}(\alpha) \right\|^2 }{ \left\langle \bm{g},\bm{p}(\alpha) \right\rangle }.
\label{eq:surrogate-1}
\end{align}
The step size $\tau_{\alpha}$ is allowed to surpass $\tau$.
\begin{lemma}
\label{lemma:1}
     $\tau_{\alpha}$ is an increasing function of $\alpha \in [0,2]$.
\end{lemma}
\begin{proof}
    See Appendix \ref{Appendix-E}.
\end{proof}
\noindent In each iteration, the update rule is 
\begin{equation}
\label{eq:surrogate-1-solver}
\begin{aligned}
\bm{x}^{k+1}  &=\operatorname{prox}_{ \tau_{\alpha^{k}} h}\left(\bm{x}^k+\bm{p}(\alpha^{k})\right) \\
& :=\arg \min _{\bm{x}} h(\bm{x})+\frac{1}{2 \tau_{\alpha^{k}}}\left\|\bm{x}-\left(\bm{x}^k+\bm{p}(\alpha^{k})\right)\right\|^2.
\end{aligned}
\end{equation}
The new iterate $\bm{x}^{k+1}$ is not assured to yield a reduced objective function value, as there is no guarantee that $m_{\alpha}(\bm{x}; \bm{x}^k)$ majorizes $q(\bm{x})$ for arbitrary $\alpha$. Now we analyze the MM condition with $\alpha$ in different ranges. Firstly, when $\alpha \in (0,1]$, $m_{\alpha}$ transitions to $m_{g}$ (where $\bm{g}_{\alpha}$ becomes $\bm{g}$ and $\tau_{\alpha}$ reduces to $\tau$). By following \eqref{eq:pg-monotone}, the surrogate $m_{\alpha}$ serves as a uniform upper bound of $q$, that is 
\begin{align*}
    m_{\alpha}(\bm{x}) \ge q(\bm{x}),
        \quad \forall \bm{x} \in \operatorname{dom}f.
\end{align*}
Then, in the case of $\alpha \in (1,2]$, the surrogate $m_{\alpha}$ majorizes $q$ when both are confined along the line connecting the current iterate and the path point. This differs from the traditional MM principle in \cite{sun2016majorization,tang2023proximal,qiu2016prime} as $m_{\alpha}$ is no longer consistently above $q$. We define the concept as \textit{opportunistic majorization} (OM).
\begin{theorem}
    \label{theorem:m-alpha-upper-bound}
    For any given $\alpha \in (1,2]$, consider the line connecting $\bm{0}$ and $\bm{p}(\alpha)$ which is given by 
    \begin{align*}
        \mathcal{X}_{\alpha} :=
        \left\{ 
            \bm{x}(\beta) := \beta \bm{p}(\alpha):~
            \forall \beta \in \mathbb{R}
        \right\}.
    \end{align*}
    Define $ \bar{q}(\beta) 
         := q(\bm{x}(\beta))$ and $
        \bar{m}_{\alpha}(\beta) 
        := m_{\alpha}( \bm{x}(\beta) )$. It holds that $\bar{q}(\beta) \le \bar{m}_{\alpha}(\beta)$ for all $\beta \in \mathbb{R}$, or equivalently, $q(\bm{x}) \le m_{\alpha}(\bm{x})$ for all $\bm{x} \in \mathcal{X}_{\alpha}$.
\end{theorem}
\begin{proof}
    See Appendix \ref{Appendix-D}.
\end{proof}
\noindent Proximal linesearch-type algorithms implicitly apply the OM without explicitly stating it \cite{bonettini2017convergence}. We formally state the concept of the OM and incorporate it into a Newton-type algorithm.
\subsection{Algorithm development}

Theorem \ref{theorem:m-alpha-upper-bound} implies that with the nontrivial surrogate function \eqref{eq:m-unscaled}, the majorization condition holds if the new iterate remains on the line, resulting a monotonically decreasing sequence $\left\{f\left(\bm{x}^k\right)\right\}_{k \in \mathbb{N}}$. To determine the largest value of $\alpha$ that makes the new iterate on the line, a backtracking procedure is needed. This procedure maximizes the contribution of second-order information to constitute the descent direction.

To facilitate the convergence analysis in the subsequent part (see the proof of Theorem \ref{the:summable_x}), we backtrack on $\alpha$ using \eqref{eq:m-unscaled}, while updating the iterate with a scaled surrogate function. In particular, the scaled surrogate function is
\begin{align}
    m_{\gamma,\alpha}(\bm{x};\bm{x}^k) 
    := 
    q(\bm{x}^k)+\left\langle \bm{g}_{\alpha},\bm{x}-\bm{x}^k \right\rangle 
    + \frac{ 1 }{2 \gamma \tau_{\alpha}}
    \left\| \bm{x} -\bm{x}^k\right\|^2\nonumber\\
    = q(\bm{x}^k)+\frac{ 1 }{2 \gamma \tau_{\alpha}} 
    \left\| \bm{x}- \left(\bm{x}^k + \bm{p}_{\gamma}(\alpha)\right)\right\|^2
    - \frac{\gamma \tau_{\alpha}}{2}
    \left\| \bm{g}_{\alpha} \right\|^2,
    \nonumber
\end{align}
where
\begin{align}
    \bm{p}_{\gamma}(\alpha) = \gamma \bm{p}(\alpha),
    \label{eq:dogleg-path-scaled}
\end{align}
and $\gamma \in (0,1)$ is a constant and typically set close to $1$ in numerical experiments. The new iterate is
\begin{equation}
\label{eq:update-rule}
\begin{aligned}
\bm{x}^{k+1}  &=\operatorname{prox}_{\gamma\tau_{\alpha^{k}} h}\left(\bm{x}^k+\bm{p}_{\gamma}(\alpha^{k})\right) \\
& :=\arg \min _{\bm{x}} h(\bm{x})+\frac{1}{2 \gamma \tau_{\alpha^{k}}}\left\|\bm{x}-\left(\bm{x}^k+\bm{p}_{\gamma}(\alpha^{k})\right)\right\|^2.
\end{aligned}
\end{equation}

The PDOM is terminated when it approaches a critical point $\bm{x}^{\star}$ where $\bm{0} \in \partial f(\bm{x}^{\star})$. From the optimality condition of the proximal operator \eqref{eq:update-rule}, it holds that 
\begin{equation}
    \label{eq:optimality-condition-1}
    \bm{0} \in \frac{1}{\gamma \tau_{\alpha^{k}}} \left(\bm{x}^{k+1}-\bm{x}^{k}-\bm{p}_{\gamma}(\alpha^{k})\right)+\partial h(\bm{x}^{k+1}).
\end{equation}
This implies 
\begin{align}
    & \partial f(\bm{x}^{k+1})
    = \nabla q(\bm{x}^{k+1})+\partial h(\bm{x}^{k+1})
    \nonumber
    \\
    &\ni \nabla q(\bm{x}^{k+1})+\frac{1}{\gamma \tau_{\alpha^{k}}} \left(\bm{x}^{k}+\bm{p}_{\gamma}(\alpha^{k})-\bm{x}^{k+1}\right)\nonumber
    \\
    &= \left(\nabla q(\bm{x}^{k+1})-\bm{g}_{\alpha^{k}}\right) - \frac{1}{\gamma \tau_{\alpha^{k}}}\left(\bm{x}^{k+1}-\bm{x}^{k}\right).
    \label{eq:stopping-criterion}
\end{align}
PDOM terminates when $\| \partial f(\bm{x}^{k+1}) \|$ is sufficiently small: 
\begin{align}
    \label{eq:stop-standard}
    \left\|\partial f(\bm{x}^{k+1})\right\| \leq 
    &\sqrt{n} \epsilon^{\mathrm{abs}}+\epsilon^{\mathrm{rel}} \max \{\left\|\nabla q(\bm{x}^{k+1})\right\|,\left\|\bm{g}_{\alpha^{k}}\right\|,\nonumber \\
    &\frac{1}{\gamma \tau_{\alpha^{k}}}\left\|\bm{x}^{k+1}\right\|, \frac{1}{\gamma \tau_{\alpha^{k}}}\left\|\bm{x}^{k}\right\|\},   
\end{align}
where $n$ is the dimension of $\bm{x}$, $\epsilon^{\mathrm{abs}}>0$ and $\epsilon^{\mathrm{rel}}>0$ are two small positive constants (motivated by \cite[Section 3.3]{boyd2011distributed}). 

This stopping criterion is different from directly using $\|\bm{x}^{k+1}-\bm{x}^{k}\|$, commonly adopted for proximal algorithms \cite{lee2014proximal}. The relationship between these two different stopping criteria can be roughly quantified by the triangle inequality 
\begin{equation}
    \left\|\partial f(\bm{x}^{k+1})\right\| 
    \leq 
    \left\|\nabla q(\bm{x}^{k+1})-\bm{g}_{\alpha^{k}}\right\|
    + \frac{1}{\gamma \tau_{\alpha^{k}}}\left\|\bm{x}^{k+1}-\bm{x}^{k}\right\|.
    \label{eq:stopping-criteria-comparison}
\end{equation}
As $1 / \tau_{\alpha^{k}}$ in \eqref{eq:stopping-criteria-comparison} can be very large, small value of $\|\bm{x}^{k+1}-\bm{x}^{k}\|$ does not necessarily imply getting close to a critical point. 

Now we formally present the PDOM in Algorithm \ref{alg:POM}. To track the largest $\alpha^{k}$, we employ a similarly straightforward strategy as presented in \cite{stella2017simple}.

\begin{algorithm}
    \caption{PDOM Algorithm}
    \label{alg:POM}
    \begin{algorithmic}[1]
    \STATE Input: $\bm{x}^{0} \in \mathbb{R}^{n}$, $\bm{Q}^{-1} \in \mathbb{R}^{n \times n}$, $\tau \in (0,1/L_{q})$, $\gamma \in (0,1)$, $\epsilon^{\mathrm{abs}}, \epsilon^{\mathrm{rel}}>0$, $k = 0$.
    \WHILE{the stopping criterion \eqref{eq:stop-standard} is not satisfied}
    \STATE Compute $\bm{x}^{k+1}$ using \eqref{eq:update-rule}, for the largest value $\alpha^{k} \in\left\{1+(1 / 2)^i \mid i \in \mathbb{N}\right\}$ such that $m_{\alpha^{k}}(\bm{x}^{k+1}; \bm{x}^k)\ge q(\bm{x}^{k+1})$.
    \STATE Compute $\bm{v}^{k+1}$ using \eqref{eq:standard-prox} and if $f(\bm{x}^{k+1})>f(\bm{v}^{k+1})$, set $\bm{x}^{k+1} = \bm{v}^{k+1}$.
    \STATE $k \gets k+1$.
    \ENDWHILE
    \STATE Output: $\bm{x}^{k}$
    \end{algorithmic}
\end{algorithm}

\begin{figure}[htbp]
    \centering
    \includegraphics[width=\linewidth]{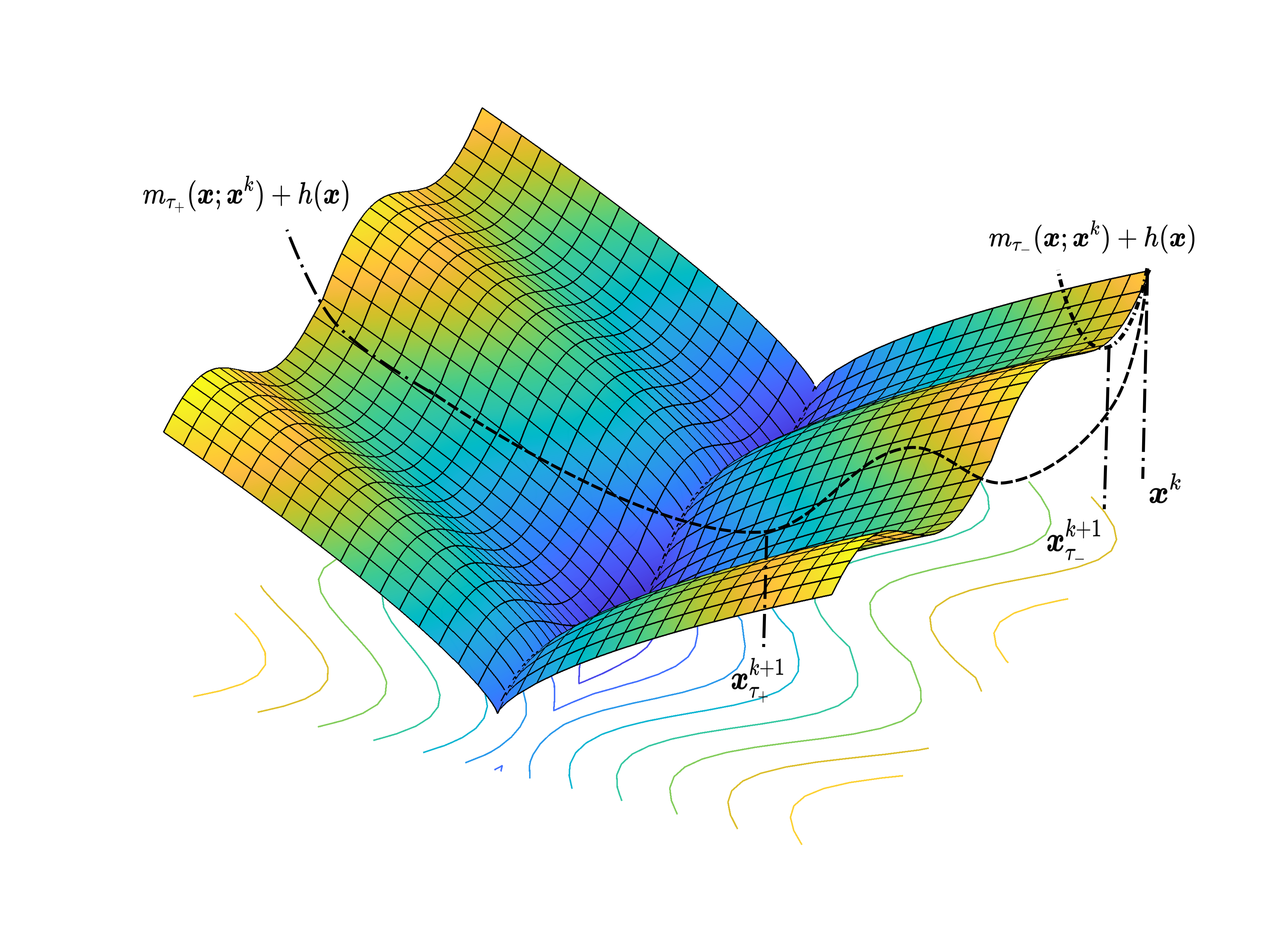}
    \caption{The configurations of global majorization ($\tau_{-}$) and opportunistic majorization ($\tau_{+}$) surrogate functions.}
    \label{fig:enter-label-2}
\end{figure}

PDOM admits the following benefits. It relaxes the global majorization to opportunistic majorization where the surrogate is only required to locally majorize the objective function. It makes the surrogate more flat and closely aligns with the characteristics of $f$ (also due to the larger step size), thus achieving a faster convergence and giving better solutions (see Figure \ref{fig:enter-label-2}). Meanwhile, it belongs to the second-order algorithm but Newton direction is readily obtained. Unlike quasi-Newton benchmarks where the Hessian approximation and Newton direction are updated at each iteration, PDOM only requires one matrix inversion. To optimize computational efficiency, a maximum searching step can be set in the backtracking process. Specifically, setting $\alpha^{k} = 1$ when $i$ exceeds certain number.

\subsection{The convergence and convergence rate analysis}
In this subsection, we establish the convergence of the iterates generated by the PDOM to a critical point of 
$f(\cdot)$ and provide the local convergence rate under the KL assumption. We first present the monotonicity of the objective function.

\begin{theorem}
\label{theo:4}
    The sequence $\left\{f(\bm{x}^k)\right\}_{k \in \mathbb{N}}$ generated by Algorithm \ref{alg:POM} is monotonically decreasing, i.e., it satisfies $f\left(\bm{x}^{k+1}\right) \le f\left(\bm{x}^{k}\right)$. 
\end{theorem}
\begin{proof}
See Appendix \ref{Appendix-G}. 
\end{proof}

We now present the main results in Theorem \ref{the:summable_x}.


\begin{theorem}
\label{the:summable_x}
    Suppose that $f$ is lower-bounded, that $q$ is a quadratic function, that $h$ is a lower semicontinuous function, that $\left\{\bm{x}^k\right\}_{k \in \mathbb{N}}$ is a sequence generated by the PDOM algorithm, then let $\left\{\bm{x}^k\right\}_{k \in \mathbb{N}}$ converge to $\bm{x}^{\star}$, we have $\bm{0} \in f(\bm{x}^{\star})$, i.e., $\bm{x}^{\star}$ is a critical point.
\end{theorem}
\begin{proof}
The proof of the theorem can be established by considering Lemma \ref{lem:tau-bound} and Lemma \ref{lem:summable_x}, following the approach outlined in \cite[Theorem 1]{li2015accelerated}.
\end{proof}

\begin{lemma}
    \label{lem:tau-bound}
    The sequence $\left\{\tau_{\alpha^k}\right\}_{k \in \mathbb{N}}$ is bounded.
\end{lemma}
\begin{proof}
    See Appendix \ref{Appendix-H}.
\end{proof}
\begin{lemma}
    \label{lem:summable_x}
    Suppose that $\left\{\bm{x}^k\right\}_{k \in \mathbb{N}}$ is a sequence generated by Algorithm \ref{alg:POM}, then it holds that
    \begin{equation}
    \label{eq:x_x_summable}
    \lim _{k \rightarrow \infty}\left\|\bm{x}^{k+1}-\bm{x}^k\right\|^2 \rightarrow 0.
    \end{equation}
\end{lemma}
\begin{proof}
    See Appendix \ref{Appendix-Y}.
\end{proof}

\label{subsec:rate-analysis}

Now, we establish the local convergence rate of the PDOM based on the KL property. We first prove $\bm{g}_{\alpha^{k}}$ becomes the gradient direction of $q$ as $k \rightarrow \infty$.

\begin{lemma}
\label{lem:relative-error}
Let $\bm{e}^k = \bm{g}_{\alpha^{k}} - \nabla q(\bm{x}^{k})$. Then, it holds that
\begin{align}
\|\partial f\left(\bm{x}^{k+1}\right)\| \le \left(\frac{1}{\gamma \tau_{\alpha^{k}}}+L_{q}\right)\left\|\left(\bm{x}^{k+1}-\bm{x}^k\right)\right\| + \|\bm{e}^k\|,
\label{eq:optimality-f-v-2}
\end{align}
and the sequence $\left\{\|\bm{e}^k\|\right\}_{k \in \mathbb{N}}$ converges to $0$ as $k \rightarrow \infty$.
\end{lemma}
\begin{proof}
    See Appendix \ref{Appendix-J}.
\end{proof}

\begin{theorem}
\label{theo:KL-global-convergence}
    Suppose that $f$ satisfies the KL property on $\omega\left(\bm{x}^k\right)$ which is the cluster point set of $\left\{\bm{x}^k\right\}_{k \in \mathbb{N}}$, then the \vspace{0.2em} sequence $\left\{\bm{x}^k\right\}_{k \in \mathbb{N}}$ generated by Algorithm \ref{alg:POM} has summable residuals, $\sum_{k=0}^{\infty}\left\|\bm{x}^{k+1}-\bm{x}^k\right\|<\infty$.
\end{theorem}
\begin{proof}
    Following the same procedure as the one in \cite[Theorem 2.9]{attouch2013convergence} and considering \eqref{eq:optimality-f-v-2} and \eqref{eq:f-x-bound}, one can easily show that the sequence $\left\{\bm{x}^k\right\}_{k \in \mathbb{N}}$ has a finite length.
\end{proof}

\begin{table*}[ht]
\centering
\begin{subtable}{1\textwidth}
\vspace{1em}
\centering
\begin{tabular}{|c|c|c|c|c|c|c|}
\hline$m$ & \begin{tabular}{c} 
$\lambda$ \\
$\times \left|\bm{A}^{T}\bm{y}\right|_{\infty}$
\end{tabular} & \begin{tabular}{c} 
PG \\
NRE/\#Iter.
\end{tabular} & \begin{tabular}{c} 
mAPG \\
NRE/\#Iter.
\end{tabular} & \begin{tabular}{c} 
PANOC \\
NRE/\#Iter.
\end{tabular} & \begin{tabular}{c} 
PANOCplus \\
NRE/\#Iter.
\end{tabular} & \begin{tabular}{c} 
PDOM \\
NRE/\#Iter.
\end{tabular} \\
\hline 
100 & \begin{tabular}{c} 
0.01 \\
0.05\\
0.10
\end{tabular} & \begin{tabular}{c} 
6.3394/533.4\\
2.57611/518.0\\
0.3002/488.4
\end{tabular} & \begin{tabular}{c} 
6.5293/843.6\\
0.6942/230.8\\
0.2908/163.2
\end{tabular} & \begin{tabular}{c} 
6.5756/101.7\\
0.7964/144.1\\
0.3034/110.7
\end{tabular} & \begin{tabular}{c} 
6.4264/119.5\\
0.3502/137.8\\
0.2916/99.7
\end{tabular} & \begin{tabular}{c} 
\textbf{9.909e-15/31.2}\\
\textbf{1.175e-10/43.4}\\
\textbf{1.453e-11/40.8}
\end{tabular}\\
\hline  
500 & \begin{tabular}{c} 
0.01 \\
0.05\\
0.10
\end{tabular} & \begin{tabular}{c} 
6.3945/1184\\
1.5339/1034\\
0.62053/417.4
\end{tabular} & \begin{tabular}{c} 
6.6079/1324\\
0.0402/336.3\\
0.5778/220.3
\end{tabular} & \begin{tabular}{c} 
6.6533/132.8\\
0.1981/275.8\\
0.5199/157.7
\end{tabular} & \begin{tabular}{c} 
6.6486/129.5\\
0.07990/248.7\\
0.6133/141.5
\end{tabular} & \begin{tabular}{c} 
\textbf{1.0495e-11/52.6}\\
\textbf{4.7643e-10/83.7}\\
0.08219/60.4
\end{tabular}\\
\hline 
1000 & \begin{tabular}{c} 
0.01 \\
0.05\\
0.10
\end{tabular} & \begin{tabular}{c} 
6.3138/$>$2000\\
0.2037/867.8\\
0.7338/412.3
\end{tabular} & \begin{tabular}{c} 
6.5169/1305.8\\
0.0671/359.6\\
0.7097/257.2
\end{tabular} & \begin{tabular}{c} 
6.5809/138.3\\
0.1507/315.5\\
0.7100/204.9
\end{tabular} & \begin{tabular}{c} 
6.5672/148.2\\
0.0974/274.6\\
0.7241/195.3
\end{tabular} & \begin{tabular}{c} 
\textbf{1.5031e-10/41.9}\\
\textbf{3.6950e-10/64.4}\\
0.3854/59.6
\end{tabular}\\
\hline
\end{tabular}
\end{subtable}
    \caption{Average recovery error and number of iterations required to reach $\|\partial f\|<10^{-5}$ for 20 independent trials across 9 instances, each characterized by different values of $m$ and $\lambda$. Sparsity level = $0.01m$.}
    \label{tab:1}
\end{table*}

\begin{theorem}
\label{theo:KL-convergence-rate}
Let $\left\{\bm{x}^k\right\}_{k \in \mathbb{N}}$ \vspace{0.1em} be any sequence generated by Algorithm \ref{alg:POM}. Suppose that $f$ satisfies the KL property on the cluster points of $\left\{\bm{x}^k\right\}_{k \in \mathbb{N}}$ with exponent $\theta \in(0,1)$, then $\left\{\bm{x}^k\right\}_{k \in \mathbb{N}}$ converges to $\bm{x}^{\star}$ such that $0 \in \partial f\left(\bm{x}^{\star}\right)$ and the following inequalities hold
\begin{enumerate}
\item for any large enough $k$, when $\theta \in\left(0, \frac{1}{2}\right)$, given any $\xi \in (0,1)$, it holds that
\begin{equation}
\label{eq:KL-convergence-rate-case-1}
\left\|\bm{x}^{k+1}-\bm{x}^{\star}\right\| \leq \xi\left\|\bm{x}^{k}-\bm{x}^{\star}\right\|^{\frac{1}{2\theta}}
\end{equation}
\item for any large enough $k$, when $\theta = \frac{1}{2}$, there exist $\kappa>0$ and $\varrho \in(0,1)$, it holds that
\begin{equation}
\label{eq:KL-convergence-rate-case-2}
\left\|\bm{x}^k-\bm{x}^{\star}\right\| \leq \sum_{i=k}^{\infty}\left\|\bm{x}^{i+1}-\bm{x}^i\right\| \leq \kappa \varrho^k
\end{equation}
\item for any large enough $k$, when $\theta \in  (\frac{1}{2},1)$, there exist $\kappa>0$, it holds that
\begin{equation}
\label{eq:KL-convergence-rate-case-3}
\left\|\bm{x}^k-\bm{x}^{\star}\right\| \leq \sum_{i=k}^{\infty}\left\|\bm{x}^{i+1}-\bm{x}^i\right\| \leq \kappa k^{\frac{1-\theta}{1-2\theta}}
\end{equation}
\end{enumerate}
\end{theorem}

\begin{proof}
    The proof technique follows the route in \cite{qian2022superlinear}. We present the proof detail in Appendix \ref{Appendix-K} for the case $\theta \in (0,\frac{1}{2})$, because the relation \eqref{eq:KL-convergence-rate-case-1} implies that the PDOM algorithm enjoys a local Q-superlinear convergence rate which differs from the local convergence rate analysis based on KL property in \cite{attouch2013convergence,stella2017forward,li2015accelerated,themelis2018forward}.
\end{proof}

The practical local convergence rate of the PDOM algorithm is determined by the value of the Łojasiewicz exponent $\theta$. To build the connection between the theoretical convergence rate analysis and the practical performance, we include the $\theta$ values for vector-sparsity promoting regularizers, e.g., $\ell_{0}$ pseudo-norm. Moreover, we derive the $\theta$ value for problems of promoting the matrix structure sparsity, e.g., RPCA, a novel addition to the existing literature.

\begin{proposition}
     The zero-norm, $\ell_{0}$, composite optimization problems have the $\theta$ value of $\frac{1}{2}$ \cite{wu2021kurdyka}. 
\end{proposition}
\begin{proposition}
\label{pro:rpca-theta}
    The RPCA problem has the $\theta$ value of $1-\frac{1}{4.9^{\upsilon}}$ where $\upsilon$ is a non-negative constant associated with the rank constraint and the dimensions of the rank-constrained matrix.
\end{proposition}
\begin{proof}
    See Appendix \ref{Appendix-L}.
\end{proof}

\section{Numerical Experiments \label{sec:simulations}}

In this section, we demonstrate numerical results on two popular applications, i.e., sparse signal recovery and robust pca, to show the fast convergence and better optimality achieved by the proposed algorithm. We compare our PDOM algorithm with some widely recognized first-order methods, namely, PG \cite{parikh2014proximal} and mAPG \cite{li2015accelerated}, as well as second-order methods, namely, PANOC \cite{stella2017simple} and PANOCplus \cite{de2022proximal} (an adaptive step size scheme for PANOC). To ensure the proposed algorithm works properly, we choose the hyper-parameters as follows. Specifically, we set 
$\tau = 1/L_{q}$, $\gamma = 0.98$, and $\epsilon^{\mathrm{abs}} = \epsilon^{\mathrm{rel}} = 10^{-12}$. To guarantee a fair comparison, each benchmark algorithm is carefully tuned. Moreover, we set the maximum iteration number to $2000$.

\subsection{Nonconvex Sparse Recovery}

The sparse recovery problem has been widely studied in signal processing and machine learning area \cite{wen2016robust,chen2014convergence,yang2020learning}. Recovering sparse solution from its noisy observation is done by solving the following optimization problem
\begin{equation}
\label{eq:LASSO-1}
\min_{\bm{x}\in \mathbb{R}^n} \frac{1}{2}\|\bm{A} \bm{x}-\bm{y}\|_2^2+\lambda\|\bm{x}\|_0,
\end{equation}
where $\bm{A}\in \mathbb{R}^{m \times n}$ and $\lambda >0$. The exact Hessian $\bm{H}$ of \eqref{eq:LASSO-1} is $\bm{A}^{T}\bm{A}$. Due to $\bm{A}$ is a fat matrix, i.e., $m < n$, $\bm{H}$ is rank deficient. To make $\bm{H}$ invertible in numerical, we add a small square norm term of $\bm{x}$ to \eqref{eq:LASSO-1}, thus we have
\begin{equation}
\label{eq:LASSO-2}
\min_{\bm{x}\in \mathbb{R}^n} \frac{1}{2}\|\bm{A} \bm{x}-\bm{y}\|_2^2+\lambda\|\bm{x}\|_0 + \frac{\mu}{2} \|\bm{x}\|^{2}_{2},
\end{equation}
where $\mu > 0$. The solution space of modified formulation \eqref{eq:LASSO-2} is nearly equivalent to that of \eqref{eq:LASSO-1} when $\mu$ is small (but the problem becomes ill-conditioned).

The experimental settings are summarized as follows. The matrix $\bm{A}$ is a standard random Gaussian matrix, with its entries following a $\mathcal{N}(0, 1)$ distribution, and has $m = n/2$ rows. The observation $\bm{y}$ is generated as $\bm{y} = \bm{A}\bm{x}^{\star} + \epsilon$ where the ground truth vector $\bm{x}^{\star}$ is sparse. Entries of $\bm{x}^{\star}$ are random positive or negative numbers. The regularization parameter $\lambda < 0.1\left|\bm{A}^{T}\bm{y}\right|_{\infty}$, following the strategy outlined in \cite{van2009probing}. The value of $\mu$ is small and makes the condition number of $\eqref{eq:LASSO-2}$ larger than $10^{5}$. All algorithms share the same randomly generated initial point $\bm{x}^{0}$ for each test. In the $k$-th iteration, we calculate the subdifferential and the normalized recovery error (NRE) for each algorithm as
\begin{equation}
    \text{NRE}(k) = \frac{\|\bm{x}^{k}-\bm{x}^{\star}\|}{\|\bm{x}^{\star}\|}.
\end{equation}

Table \ref{tab:1} provides an overview of the average performance of the proposed algorithm and benchmark methods across various problem sizes and regularization parameters (noiseless). The results indicate that PDOM has a faster convergence and the ability to find a better local optimum than the other algorithms. This is evident in its reduced number of iterations to approach the critical point and ultimately achieve a smaller NER. Note that in two instances without bolded NRE, PDOM achieve larger errors compared to the other instances due to the choice of $\lambda$, resulting in the local optimal point being distant from the ground truth. Despite this, the PDOM still outperforms the benchmark methods.

\begin{table}[htbp]
\centering
\begin{tabular}{|c|c|c|c|}
\hline \diagbox{$\lambda$}{$m$} & 100  & 500 & 1000 \\
\hline 
0.01 & 209.1& 486.3& 216.2\\
\hline 
0.05 & 368.2& 339.9&343.6\\
\hline 
0.10 & 351.4&468.7 &443.5\\
\hline
\end{tabular}
    \caption{Average number of proximal operators PDOM needed to reach $\|\partial f\|<10^{-5}$ across 9 instances in Table \ref{tab:1}. The scale of $\lambda$ is $\times \left|\bm{A}^{T}\bm{y}\right|_{\infty}$}
    \label{tab:3}
\end{table}

The average number of operator calculations required by PDOM is summarized in Table \ref{tab:3}. It's worth noting that PG, PANOC, and PANOCplus each operate one proximal mapping, while mAPG requires two in each iteration. The backtracking scheme in PDOM incurs a higher cost, resulting in more proximal operator calculations compared to other methods, as a trade-off for achieving fewer iterations.

\begin{figure}[htbp]
\begin{minipage}[b]{.48\linewidth}
  \centering
\centerline{\includegraphics[width=4.5cm]{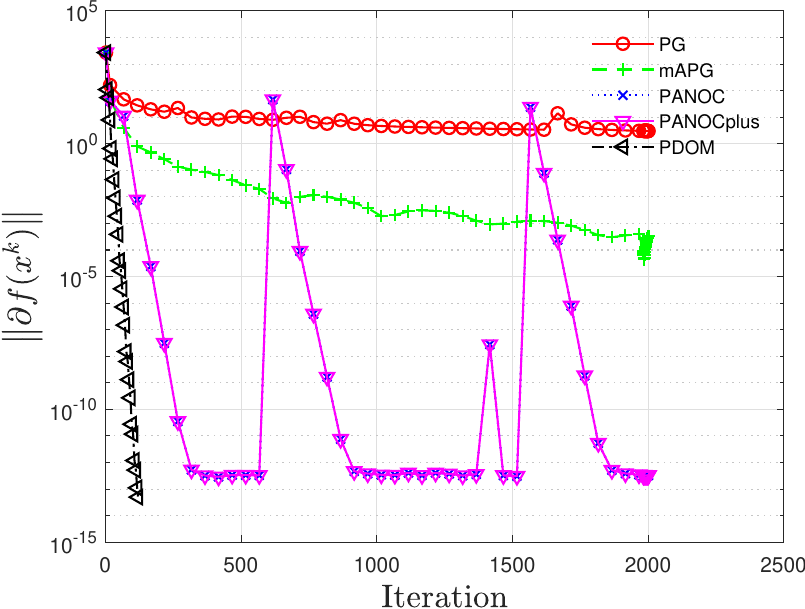}}
\end{minipage}
\hfill
\begin{minipage}[b]{0.48\linewidth}
  \centering
\centerline{\includegraphics[width=4.5cm]{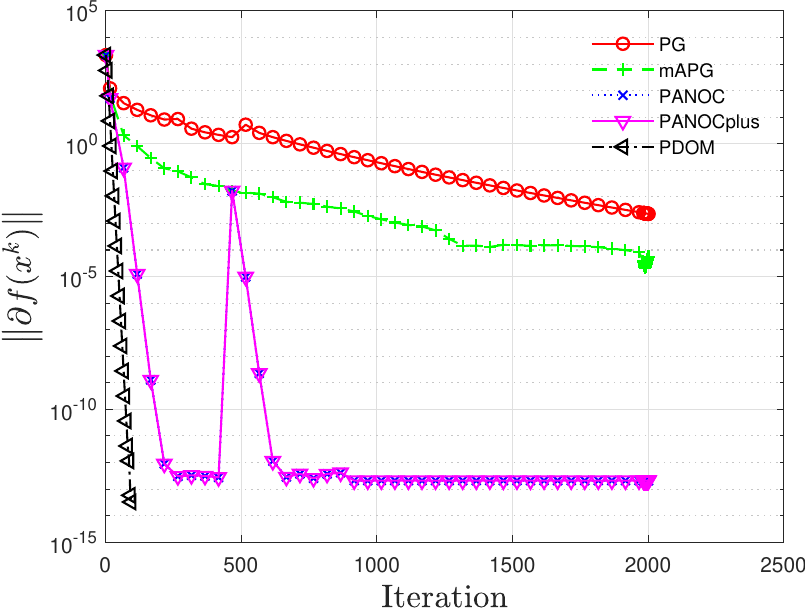}}
\end{minipage}
\begin{minipage}[b]{.48\linewidth}
  \centering
\centerline{\includegraphics[width=4.5cm]{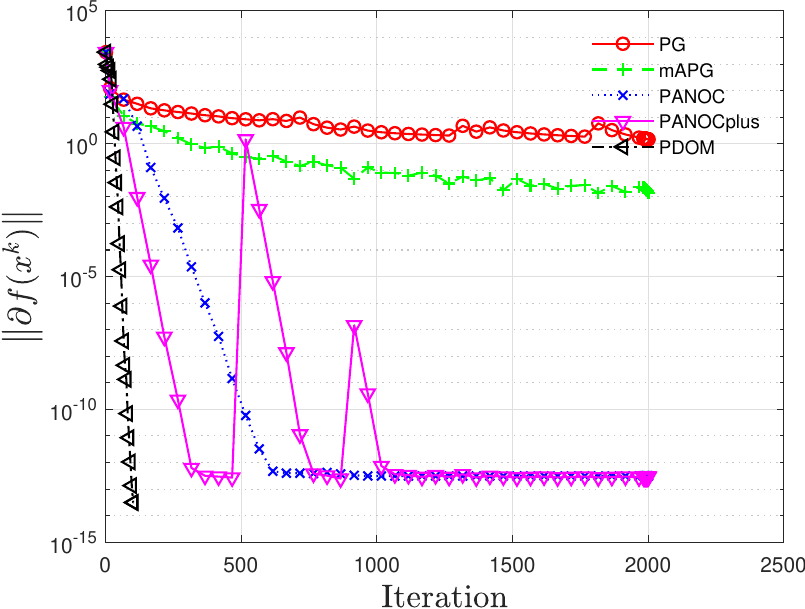}}
\end{minipage}
\hfill
\begin{minipage}[b]{0.48\linewidth}
  \centering
\centerline{\includegraphics[width=4.5cm]{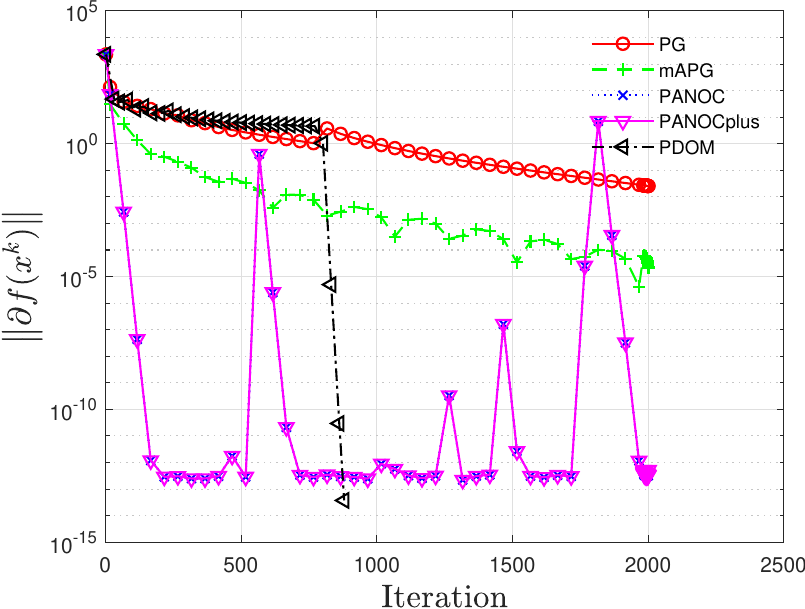}}
\end{minipage}
\caption{Convergence behavior of subdifferential for the first four instances (one realization) from Table \ref{tab:1}.}
\label{fig:1}
\end{figure}

Figure \ref{fig:1} depicts the specific convergence behavior of the compared algorithms on four instances of results in Table \ref{tab:1}. It can be seen that PDOM usually outperforms other baseline algorithms on convergence rate. The last subfigure seemingly shows PDOM a slower convergence rate, but PANOC(plus) experiences multiple sharp ascents after first converge, referring to meet long flat regions or saddle points. PDOM undergoes no oscillation. Combining the results in Table \ref{tab:1}, PDOM has the advance in finding a better local region.

In Figure \ref{fig:l0_pt}, we plot the phase transition curves for the tested algorithms. The outcomes demonstrate a significantly higher success recovery rate for the PDOM algorithm compared to the benchmark algorithms in both noiseless and noisy cases.

\begin{figure}
    \centering
    \begin{subfigure}{0.327\textwidth}
        \includegraphics[width=\linewidth]{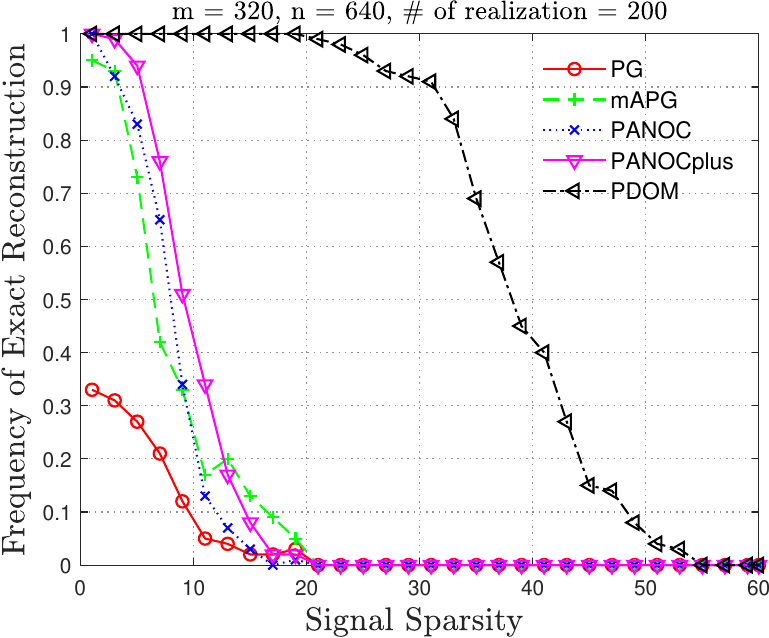}
        \caption{Noiseless case.}
    \end{subfigure}
    \begin{subfigure}{0.327\textwidth}
        \includegraphics[width=\linewidth]{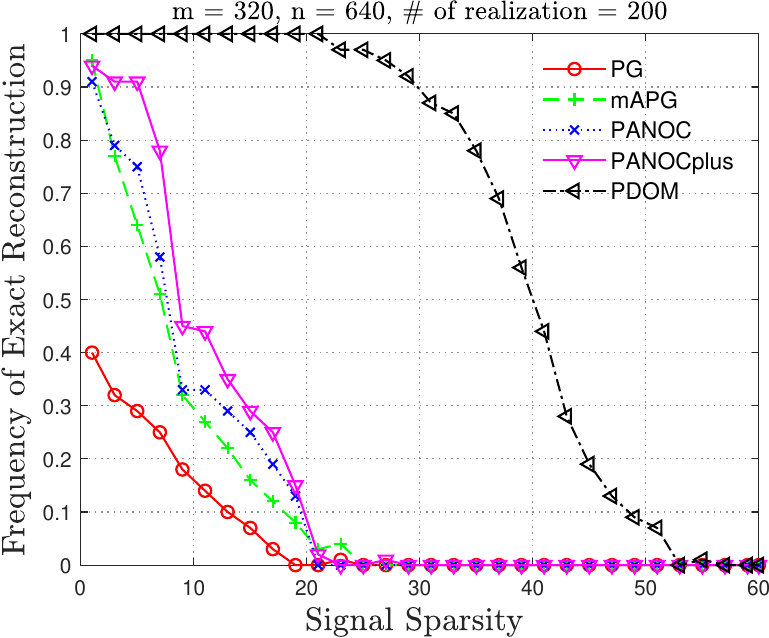}
        \caption{Noisy case: $\epsilon \sim \mathcal{N}(0, 0.1)$.}
    \end{subfigure}
    \caption{Phase transition curve of $\ell_{0}$ sparse recovery at varying sparsities. Realizations with random initialization
are considered successful if NRE $<10^{-4}$ for noiseless case or $< 10^{-2}$ for noisy case.}
    \label{fig:l0_pt}
\end{figure}

\subsection{Robust Principal Component Analysis (RPCA)}

We consider the problem of matrix decomposition, aiming to decompose a given matrix $\bm{M} \in \mathbb{R}^{m\times n}$ into a low-rank matrix $\bm{L}$ and a sparse matrix $\bm{S}$. This problem finds applications in detecting anomalies in traffic volume \cite{mardani2012dynamic} and identifying moving objects \cite{cao2015total}. The optimization problem is formulated as follows
\begin{equation}
\underset{\bm{L}, \bm{S} \in \mathbb{R}^{m \times n}}{\min} \frac{1}{2}\left\|\bm{M}-\bm{L}-\bm{S}\right\|_F^2+\delta_{\text {rank } \leq r}\left(\bm{L}\right)+\lambda\left\|\bm{S}\right\|_0,
\label{eq:RPCA}
\end{equation}
where $\delta_{\text {rank} \leq r}\left(\bm{L}\right)$ denotes the
indicator function which returns zero when $\text {rank}(\bm{L})\leq r$, and $+\infty$ otherwise. The nonconvexity of \eqref{eq:RPCA} arises from the indicator function associated with the rank constraint and the $\ell_{0}$ pseudo-norm. The proximal mapping of $\delta_{\text {rank } \leq r}\left(\bm{L}\right)$ is the projection onto the at most rank-$r$ space, i.e., $\Pi_{\text {rank } \leq r}(\bm{L})=\bm{U}_r \operatorname{diag}\left(\sigma_1, \ldots, \sigma_r\right) \bm{V}_r^T$, where $\sigma_1 \ldots \sigma_r$ are $r$ largest singular values of $\bm{L}$, and $\bm{U}_r, \bm{V}_r$ denote the matrices of left and right singular vectors, respectively. 

The experimental settings follow the one in \cite{candes2011robust}. In particular, one has $m = n$, $r = 0.05m$, and $\lambda=1 / \sqrt{n}$. We generate a rank-$r$ matrix $\bm{L}^{\star}$ with entries following a $\mathcal{N}(0, 1)$ distribution. The sparse matrix $\bm{S}^{\star}$ contains $0.1m^{2}$ independent Bernoulli $\pm 1$ entries, whose locations are randomly selected. The initial points, $\bm{L}^{0}$ and $\bm{S}^{0}$, are randomly generated. 

\begin{table}
    \centering
    \begin{tabular}{|c|c|c|c|}
\hline
$m$ & 100 & 500 & 1000\\
\hline
\begin{tabular}{c} 
PG \\
NER/\#Iter
\end{tabular} & 0.048/$>$2000 & 0.212/$>$2000 & 0.142/$>$2000\\
\hline
\begin{tabular}{c} 
mAPG \\
NER/\#Iter
\end{tabular} & 0.108/427 & 0.2980/483 & 0.191/511\\
\hline
\begin{tabular}{c} 
PANOC \\
NER/\#Iter
\end{tabular} & 0.081/73 & 0.263/79 & 0.556/77\\
\hline
\begin{tabular}{c} 
PANOCplus \\
NER/\#Iter
\end{tabular} & 0.116/114 & 0.253/128 & 0.188/120\\
\hline
\begin{tabular}{c} 
PDOM \\
NER/\#Iter
\end{tabular} & \textbf{5.308e-14/43} & \textbf{0.0027/54} &  \textbf{0.0099/42} \\
\hline
\end{tabular}
    \caption{Average recovery error of the low-rank matrix and number of iterations to reach $\|\partial f\|<10^{-5}$ for 10 independent trials with different $m$.}
    \label{tab:2}
\end{table}

The average performances of the PDOM and benchmark algorithms for solving RPCA with different sizes are presented in Table \ref{tab:2}. The results indicate that PDOM has a faster convergence rate and finds better local minimizers. 

Figure \ref{fig:2} demonstrates the convergence behavior of one instance. It can be seen that PDOM converges more rapidly (but still at a linear rate) to a critical point compared to the benchmarks and it does so without experiencing further oscillations. Meanwhile, PDOM attains a much smaller recovery error of the low-rank matrix.

\begin{figure}
\begin{minipage}[b]{.48\linewidth}
  \centering
\centerline{\includegraphics[width=4.5cm]{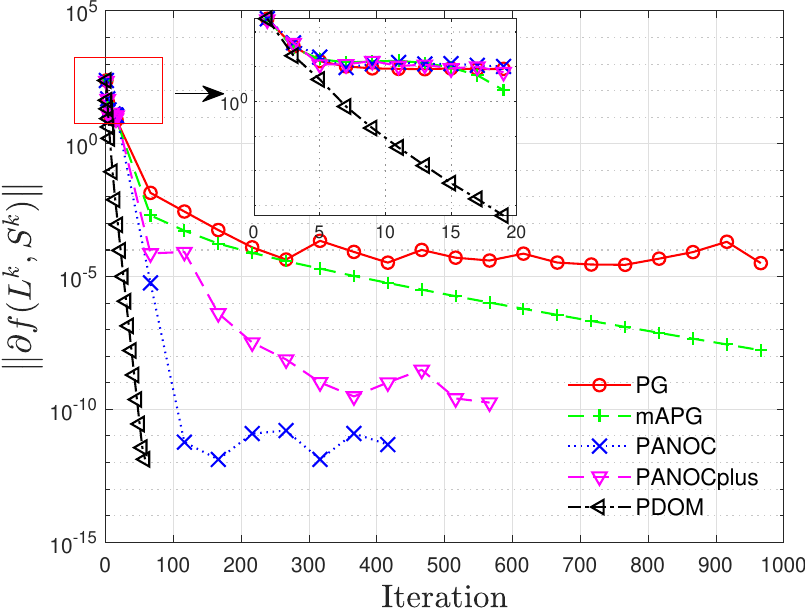}}
\end{minipage}
\hfill
\begin{minipage}[b]{0.48\linewidth}
  \centering
\centerline{\includegraphics[width=4.5cm]{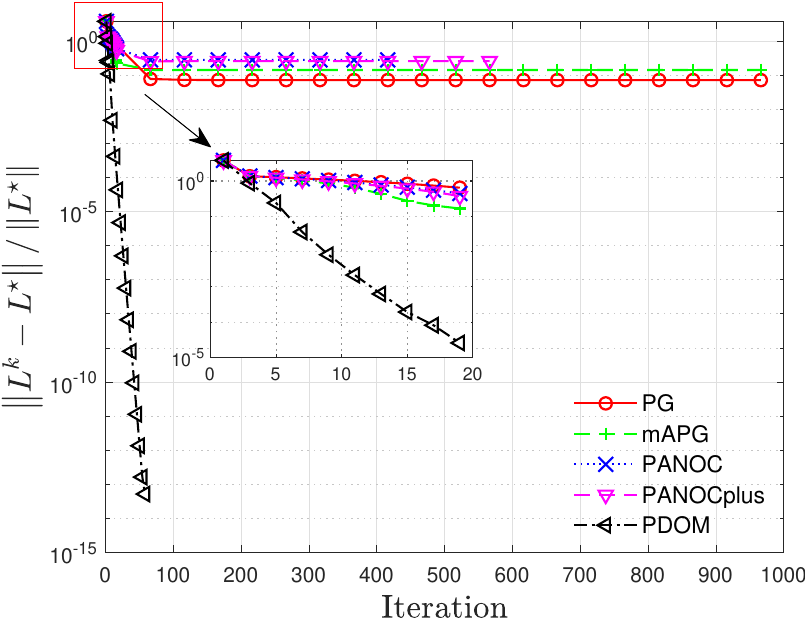}}
\end{minipage}
\caption{Convergence behavior of the RPCA problem with $m = 100$. \textit{Left}: Performance comparisons of subdifferential. \textit{Right}: Performance comparisons of normalized recovery error of the low-rank matrix}
\label{fig:2}
\end{figure}

We further plots the fraction of low-error recoveries for varying ranks in Figure \ref{fig:lr_pt}. Note that the PDOM algorithm successfully recovers $\bm{L}$ over a much wider rank range with a higher possibility. 

\begin{figure}[htbp]
    \centering
    \begin{subfigure}{0.33\textwidth}
        \includegraphics[width=\linewidth]{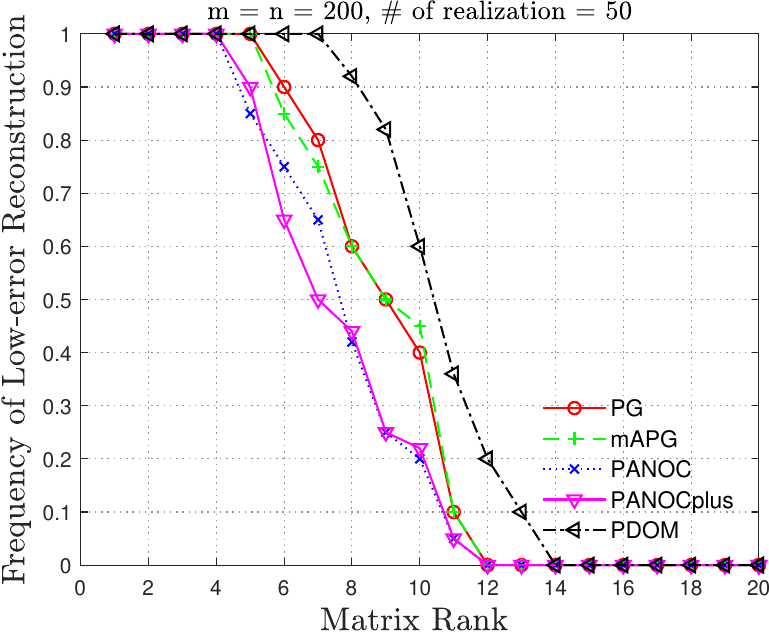}
        \caption{Sparsity level of $\bm{S}^{\star}$ = $0.05m^{2}$ }
    \end{subfigure}
    \begin{subfigure}{0.33\textwidth}
        \includegraphics[width=\linewidth]{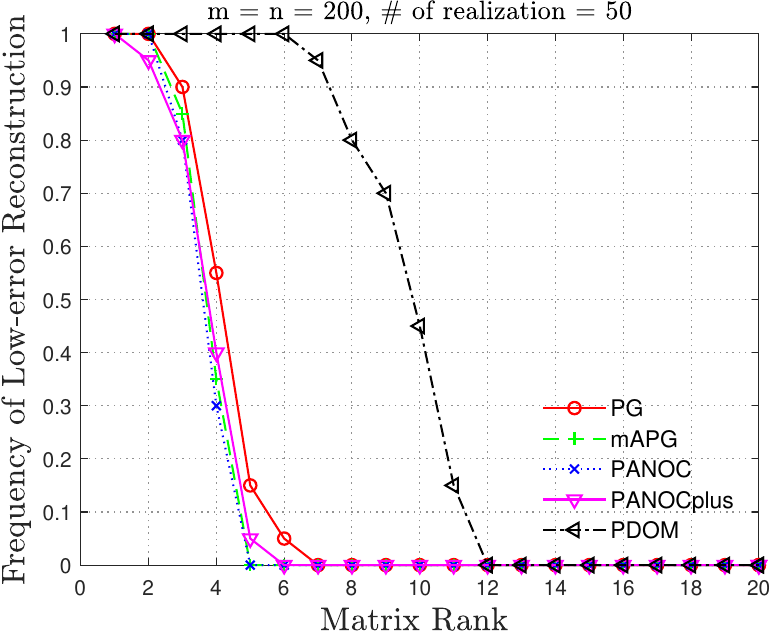}
        \caption{Sparsity level of $\bm{S}^{\star}$ = $0.2m^{2}$.}
    \end{subfigure}
    \caption{Phase transition curve of RPCA at varying ranks. Realizations with random initialization
are considered successful if $\left\|\hat{\bm{L}}-\bm{L}^\star\right\|_F /\left\|\bm{L}^\star\right\|_F<10^{-3}$.}
    \label{fig:lr_pt}
\end{figure}

\section{Conclusion}

Our paper introduces the PDOM algorithm for efficient handling of nonconvex and nonsmooth problems with a quadratic term. In each iteration, the algorithm builds and minimizes a majorization along a hybrid direction. Theoretical analysis establishes the global convergence of the PDOM to a critical point, and local convergence rates are explored based on the KL property. Numerical experiments verify faster convergence and the ability to reach superior local optimum for nonconvex problems.
\appendices
\section{Preliminaries of Proofs}
 Given a positive definite matrix $\bm{Q} \succ 0$ with bounded eigenvalues, it holds that
\begin{align}
    & \lambda_{\min} \| \bm{g} \|^2 
    \le \| \bm{g} \|_{\bm{Q}}^2 = \bm{g}^{T} \bm{Q} \bm{g} 
    \le \lambda_{\max} \| \bm{g} \|^2,\label{eq:lambda--inequality}
\end{align}
where $\lambda_{\max}$ and $\lambda_{\min}$ denote the largest and the smallest eigenvalue of $\bm{Q}$, respectively. To facilitate subsequent analyses, a coordinate shift is applied, making $\left(\bm{x}^k, q\left(\bm{x}^k\right)\right)$ the new origin. The smooth part and surrogate become
\begin{align}
q(\bm{x}):=\left\langle\bm{g}, \bm{x}\right\rangle+\frac{1}{2 }\|\bm{x}\|^2_{\bm{Q}}, \quad
m_{\alpha}(\bm{x}):=\left\langle\bm{g}_\alpha, \bm{x}\right\rangle+\frac{1}{2 \tau_\alpha}\|\bm{x}\|^2 .\nonumber
\end{align}

\section{Proof of Lemma \ref{lem:zero-inner-product}}
\label{Appendix-B}
We begin with the trivial case where $\alpha \in(0,1]$,
\begin{align}
    & \quad \langle \bm{p}(\alpha),\nabla q(\bm{p}(\alpha)) \rangle \nonumber\\
    & \quad =-\tau\bm{g}\left(-\tau\bm{Q}\bm{g}+\bm{g}\right) = \tau^{2}\left(\| \bm{g} \|_{\bm{Q}}^2 - \frac{1}{ \tau}\| \bm{g} \|^2\right)\nonumber\\
    & \quad \le \tau^{2}\left(\lambda_{\max}\| \bm{g} \|^2 - \frac{1}{ \tau}\| \bm{g} \|^2\right) \le 0.\nonumber
\end{align}
Then for the remaining range of $\alpha \in \left( 1,2 \right]$, we have
\begin{align}
    &\langle \bm{p}(\alpha),\nabla q(\bm{p}(\alpha)) \rangle \nonumber\\
    &= -\tau\|\bm{g}\|^2+(\alpha-1)\left(\tau\|\bm{g}\|^2-\|\bm{g}\|_{\bm{Q}^{-1}}^2\right)\nonumber\\
    & < (\alpha-1)\left(\tau\|\bm{g}\|^2-\|\bm{g}\|_{\bm{Q}^{-1}}^2\right) \le 0.
    \label{eq:inner-product-path-qd-12}
\end{align}
The equality of \eqref{eq:inner-product-path-qd-12} only holds when $\tau$ is exactly $1/\lambda_{\max}$.
\section{Proof of Theorem \ref{theorem:m-alpha-upper-bound}}
\label{Appendix-D}
We start the proof by showing the following lemmas.
\begin{lemma}
\label{lem:1}
    Given $\bar{q}$ and $\bar{m}_{\alpha}$ defined in Theorem \ref{theorem:m-alpha-upper-bound}, it holds that $\bar{q}(0) = \bar{m}_{\alpha}(0)$ and $\bar{q}^{\prime}(0) = \bar{m}^{\prime}_{\alpha}(0) < 0$.
\end{lemma}
\begin{proof}
    It is easy to show that $\bar{q}(0)= q(\bm{x}(0)) = q(\bm{0}) = 0$, and $\bar{m}_{\alpha}(0) = m_{\alpha}(\bm{x}(0)) = m_{\alpha}(\bm{0}) = 0$. Then we prove the negative gradient.
    It holds that 
\begin{align*}
    &\bar{q}^{\prime}(0) = \left. \beta \bm{p}(\alpha)^{\mathsf{T}} \bm{Q} \bm{p}(\alpha) + \bm{g}^{\mathsf{T}} \bm{p}(\alpha) \right|_{\beta=0} =\langle \bm{g},\bm{p}(\alpha) \rangle,\\
    &\bar{m}_{\alpha}^{\prime}(0)= \left. \bm{g}_{\alpha}^{\mathsf{T}} \bm{p}(\alpha) + \frac{1}{\tau_{\alpha}} \beta \| \bm{p}(\alpha) \|^2  \right|_{\beta=0} = \langle \bm{g},\bm{p}(\alpha) \rangle.
\end{align*}
From Lemma \ref{lem:zero-inner-product}, we prove that  $\bar{q}^{\prime}(0) = \bar{m}_{\alpha}^{\prime}(0) < 0$.
\end{proof}
\begin{lemma}
\label{lem:uni-variate-quadratic-bound}
Let $m(x)$ and $q(x)$ be univariate strictly convex quadratic functions. Suppose that 
\begin{enumerate}[label=(\roman*)]
    \item $m(0) = q(0)$ and $m^{\prime}(0)  = q^{\prime}(0) \ne 0$,
    \item $x_m^{\#} = \tau x_q^{\#}$ for some $\tau \in (0,1)$, where $x_m^{\#}  := \arg \min_x~ m(x)$ and $
        x_q^{\#} 
            := \arg \min_{x}~ q(x)$,
\end{enumerate}
then, it holds that
\begin{enumerate}[label=(\roman*)]
    \item $\left| \frac{ x_m^{\#} }{ m^{\prime}(0) } \right| < \left| \frac{ x_q^{\#} }{ q^{\prime}(0) } \right|$,
    \item $m(x) \ge q(x)$ for all $x$, and the equality holds if and only if $x=0$.
\end{enumerate}
\end{lemma}
\begin{proof}
    As both $m(x)$ and $q(x)$ are quadratic, one can write them as $q(x) = q(0) + q^{\prime}(0) x + \frac{1}{2 \tau_q} x^2$ and $m(x) = q(0) + q^{\prime}(0) x + \frac{1}{2 \tau_m} x^2$. It is clear that $x_q^{\#} 
        = - \tau_q q^{\prime}(0),
        x_m^{\#} 
        = - \tau_m q^{\prime}(0).$
    From the assumption that 
    \begin{align*}
        & x_m^{\#}
        = \arg \min_x~ m(x) 
        = \tau \arg \min_x~ q(x)
        = \tau x_q^{\#},
    \end{align*}
    it holds that 
    \begin{align*}
        & 
        \left| \frac{ x_m^{\#} }{ m^{\prime}(0) } \right| 
        = \tau_m = \tau \tau_f 
        < \tau_f 
        = \left| \frac{ x_q^{\#} }{ q^{\prime}(0) } \right|,
    \end{align*}
    or equivalently, $\frac{1}{\tau_m} > \frac{1}{\tau_f}$. Therefore, $m(x) \ge q(x)$ where the equality holds if and only if $x=0$. Both claims in the lemma are therefore proved. 
\end{proof}
Based on Lemma \ref{lem:1} and Lemma \ref{lem:uni-variate-quadratic-bound}, Theorem \ref{theorem:m-alpha-upper-bound} can be proved by showing the lemma below.
\begin{lemma}
    \label{lem:m-bar-f-bar}
    Considering $\bar{q}(\beta)$ and $\bar{m}(\beta)$ defined in Theorem \ref{theorem:m-alpha-upper-bound}, it holds that $1 = \arg \min_{\beta} \bar{m}_{\alpha}(\beta)
        \le \arg \min_{\beta} \bar{q}(\beta)$.
\end{lemma}
\begin{proof}
    It is established that
    \begin{align*}
        &\bar{m}_{\alpha}^{\prime}(1) = \left. \bm{g}_{\alpha}^{\mathsf{T}} \bm{p}(\alpha) + \frac{1}{\tau_{\alpha}} \beta \| \bm{p}(\alpha) \|^2  \right|_{\beta=1}
        = 0.
    \end{align*}
    The claim that  $1 = \arg \min_{\beta} \bar{m}_{\alpha}(\beta)$ is therefore proved. We now show that $\bar{q}^{\prime}(1) \le 0$. It is clear that 
    \begin{align*}
        &\bar{q}^{\prime}(1)
        = \left. \beta \bm{p}(\alpha)^{\mathsf{T}} \bm{Q} \bm{p}(\alpha) + \bm{g}^{\mathsf{T}} \bm{p}(\alpha) \right|_{\beta=1}\le 0,
    \end{align*}
where the last inequality comes from Lemma \ref{lem:zero-inner-product}. Combining this with Lemma \ref{lem:uni-variate-quadratic-bound} that $\bar{f}^{\prime}(0) = \bar{m}^{\prime}_{\alpha}(0) < 0$, it can be concluded that $1 = \arg \min_{\beta} \bar{m}_{\alpha}(\beta)
        \le \arg \min_{\beta} \bar{q}(\beta)$.
This completes the proof.
\end{proof}
\section{Proof of Lemma \ref{lemma:1}}
\label{Appendix-E}
Tthe step size $\tau_{\alpha}$ is said to be an increasing function of $\alpha \in [0,2]$ if $\frac{d}{d\alpha} \tau_{\alpha} \ge 0$ hold. The positive gradient can be proved with simple algebra.
\section{Proof of Theorem \ref{theo:4}}
\label{Appendix-G}

It holds that
    \begin{align*}
    &f(\bm{x}^k) = h(\bm{x}^k) +  m_{\gamma,\alpha}(\bm{x}^k;\bm{x}^k) \\
    & \overset{(a)}{\ge} h(\bm{x}^{k+1})  +  m_{\gamma,\alpha}(\bm{x}^{k+1};\bm{x}^k)\overset{(b)}{\ge} h(\bm{x}^{k+1})  +  m_{\alpha}(\bm{x}^{k+1};\bm{x}^k)\\
    & \overset{(c)}{\ge} h(\bm{x}^{k+1}) + q(\bm{x}^{k+1}) = f(\bm{x}^{k+1}),
\end{align*}
where $(a)$ is because of the proximal operator, $(b)$ holds due to $\gamma<1$, and $(c)$ is because of the backtracking rule.

\section{Proof of Lemma \ref{lem:tau-bound}}
\label{Appendix-H}
Based on the definition of $\tau_{\alpha}$, it holds that
\begin{align}
    \tau_{\alpha^k} &=  -\frac{ \left\| \bm{p}(\alpha^{k}) \right\|^2 }{ \left\langle \nabla q(\bm{x}^{k}),\bm{p}(\alpha^{k}) \right\rangle }\nonumber\\
    & = \frac{ \left\| \bm{p}(\alpha^{k}) \right\|^2 }{(\alpha^{k}-2)\tau\left\|\nabla q(\bm{x}^{k})\right\|^{2}+(1-\alpha^{k})\left\|\nabla q(\bm{x}^{k})\right\|^{2}_{ \bm{Q}^{-1}}} \nonumber.
\end{align}
Since $\bm{Q}$ has bounded eigenvalues, it follows that both the numerator and denominator remain bounded. Consequently, we prove that the sequence $\left\{\tau_{\alpha^k}\right\}_{k \in \mathbb{N}}$ is bounded.

\section{Proof of Lemma \ref{lem:summable_x}}
\label{Appendix-Y}
 By following \eqref{eq:update-rule}, the path search procedure finds a new update $\bm{x}^{k+1}$ to make $m_{\alpha^{k}}(\bm{x}^{k+1}; \bm{x}^k)$ an upper bound of $q(\bm{x}^{k+1})$, thus we have
    \begin{align}
    &h(\bm{x}^{k}) 
    = h(\bm{x}^{k}) + \left\langle \bm{g}_{\alpha^k},\bm{x}^{k} - \bm{x}^k \right\rangle 
    + \frac{1}{2 \gamma \tau_{\alpha^k}} \left\| \bm{x}^{k} - \bm{x}^k \right\|^2 \nonumber \\
    &
    \ge h(\bm{x}^{k+1}) + \left\langle \bm{g}_{\alpha^k},\bm{x}^{k+1} - \bm{x}^k \right\rangle + \frac{1}{2 \gamma \tau_{\alpha^k}} \left\| \bm{x}^{k+1} - \bm{x}^k \right\|^2\label{eq:h-x}\\
    &q(\bm{x}^{k+1})
    \label{eq:q-x}
    \le q(\bm{x}^k)+\left\langle \bm{g}_{\alpha^k},\bm{x}^{k+1} - \bm{x}^k \right\rangle 
    + \frac{1}{2 \tau_{\alpha^k}} \left\| \bm{x}^{k+1} - \bm{x}^k \right\|^2.
\end{align}
Combining \eqref{eq:h-x} and \eqref{eq:q-x}, we derive
\begin{equation}
    \label{eq:f-x-bound}
    f(\bm{x}^{k+1}) \le f(\bm{x}^{k})-\left(\frac{1}{2 \gamma \tau_{\alpha^k}}-\frac{1}{2 \tau_{\alpha^k}}\right)\left\|\bm{x}^{k+1} - \bm{x}^k\right\|^{2}.
\end{equation}
Since we assume that $f$ is bounded from below, then the sequence $\left\{\bm{x}^k\right\}_{k \in \mathbb{N}}$ is bounded and has a limiting point.  We use $\bm{x}^{\star}$ and $f^{\star}$ to denote the limiting point of $\left\{\bm{x}^k\right\}_{k \in \mathbb{N}}$ and the objective function value on that point, respectively. By summing over $k=1,2, \cdots, \infty$, we obtain that
\begin{align*}
\sum_{k=1}^{\infty}\left(\frac{1}{2 \gamma \tau_{\alpha^k}}-\frac{1}{2 \tau_{\alpha^k}}\right)\left\|\bm{x}^{k+1}-\bm{x}^k\right\|^2 \leq f\left(\bm{x}^1\right)-f^{\star}<\infty.
\end{align*}
Given that $\left\{\tau_{\alpha^k}\right\}_{k \in \mathbb{N}}$ is bounded, we derive that
\begin{equation}
\label{eq:the5-1}
\lim _{k \rightarrow \infty}\left\|\bm{x}^{k+1}-\bm{x}^k\right\|^2 \rightarrow 0.
\end{equation}
Given that $\bm{x}^{k+1}$ becomes $\bm{v}^{k+1}$ if $\bm{v}^{k+1}$ leads to a smaller objective function, it is also necessary to satisfy
\begin{align}
\label{eq:v_x_summable}
\lim _{k \rightarrow \infty}\left\|\bm{v}^{k+1}-\bm{x}^k\right\|^2 \rightarrow 0.
\end{align}
For the proof of this statement, refer to \cite{frankel2015splitting}. This concludes the proof.

\section{Proof of Lemma \ref{lem:relative-error}}
\label{Appendix-J}
By considering the optimality condition of \eqref{eq:update-rule}, we have that
\begin{align*}
\left\|\bm{g}_{\alpha^{k}}+\frac{1}{\gamma\tau_{\alpha^{k}}}\left(\bm{x}^{k+1}-\bm{x}^{k}\right)-\nabla q(\bm{x}^{k+1})\right\| \in \left\| \partial f(\bm{x}^{k+1}) \right\|.
\end{align*}
By triangle inequality and smoothness of $\nabla q$, we have
\begin{align*}
    &\left\| \partial f(\bm{x}^{k+1}) \right\| \le \left\| \bm{g}_{\alpha^{k}}-\nabla q(\bm{x}^{k+1})\right\|+ \frac{1}{\gamma\tau_{\alpha^{k}}}\left\| \bm{x}^{k+1}-\bm{x}^{k}\right\|\\
    &\le \left\| \nabla q(\bm{x}^{k})-\nabla q(\bm{x}^{k+1})\right\|+ \frac{1}{\gamma\tau_{\alpha^{k}}}\left\| \bm{x}^{k+1}-\bm{x}^{k}\right\| + \left\|\bm{e}^{k}\right\|\\
    & \le \left(L_{q} + \frac{1}{\gamma\tau_{\alpha^{k}}}\right)\left\| \bm{x}^{k+1}-\bm{x}^{k}\right\| + \left\|\bm{e}^{k}\right\|.
\end{align*}
Based on \eqref{eq:the5-1} and Theorem \ref{the:summable_x}, we have $\lim_{k \rightarrow \infty} \left\|\bm{e}^{k}\right\| \rightarrow 0$.

\section{Proof of Theorem \ref{theo:KL-convergence-rate}}
\label{Appendix-K}
    By considering the assumption that $f$ has the KL property, for large enough $k$, based on \eqref{eq:optimality-f-v-2}, we have
    \begin{align*}
    &\psi^{\prime}\left(f\left(\bm{x}^k\right)-f\left(\bm{x}^{\star}\right)\right)\\
    &\geq \frac{1}{\left(\frac{1}{\gamma \tau_{\alpha^{k-1}}}+L_{q}\right)\left\|\left(\bm{x}^{k}-\bm{x}^{k-1}\right)\right\| + \|\bm{e}^{k-1}\|}.\nonumber
     \end{align*}
    Recall that the desingualarizing function has form of $\psi(t) = \frac{C}{1-\theta}t^{1-\theta}$ for $t \in[0, \infty)$, we have
    \begin{align}
    &\left(f\left(\bm{x}^k\right)-f\left(\bm{x}^{\star}\right)\right)^{\theta} \nonumber \\
    &\leq C \left(\left(\frac{1}{\gamma \tau_{\alpha^{k-1}}}+L_{q}\right)\left\|\left(\bm{x}^{k}-\bm{x}^{k-1}\right)\right\| + \|\bm{e}^{k-1}\|\right).
    \label{eq:KL-objective-bound-by-residual}
    \end{align}
    Due to $f\left(\bm{x}^k\right)-f\left(\bm{x}^{k+1}\right) \leq f\left(\bm{x}^k\right)-f(\bm{x}^{\star})$ for all $k$, by combining \eqref{eq:f-x-bound} and \eqref{eq:KL-objective-bound-by-residual}, we have
    \begin{equation}
    \left\|\bm{x}^{k+1}-\bm{x}^k\right\| \leq\left(\mathcal{B}\right)^{\frac{1}{2\theta}}  \left\|\bm{x}^k-\bm{x}^{k-1}\right\|^{\frac{1}{2\theta}}+ \left(\mathcal{C}\right)^{\frac{1}{2\theta}}\|\bm{e}^{k-1}\|^{\frac{1}{2\theta}}\nonumber,
    \end{equation}
    where
    \begin{align*}
        \mathcal{B} = \frac{C \left(\frac{1}{\gamma \tau_{\alpha^{k-1}}}+L_{q}\right)}{ \left(\frac{1}{2 \gamma \tau_{\alpha^{k}}}-\frac{1}{2 \tau_{\alpha^{k}}}\right)^\theta}\text{,} \quad \mathcal{C} = \frac{C}{ \left(\frac{1}{2 \gamma \tau_{\alpha^{k}}}-\frac{1}{2 \tau_{\alpha^{k}}}\right)^\theta}.
    \end{align*}
    Assume that there exists $j > k$, then by recursion law, it holds that
    \begin{align}
    &\left\|\bm{x}^{j+1}-\bm{x}^j\right\|\nonumber\leq \left(\mathcal{B}\right)^{\frac{1}{2\theta}\sum_{t=0}^{j-k}{\frac{1}{2\theta}}^t}   \left\|\bm{x}^k-\bm{x}^{k-1}\right\|^{{\frac{1}{2\theta}}^{j-k+1}}\\
    & \quad \quad \quad \quad \quad \quad+ \left(\mathcal{C}\right)^{\frac{1}{2\theta}\sum_{t=0}^{j-k}{\frac{1}{2\theta}}^t}   \left\|\bm{e}^{k-1}\right\|^{{\frac{1}{2\theta}}^{j-k+1}}\nonumber\\
    & = \left(\mathcal{B}\right)^{\frac{1}{2\theta}}  \left\|\bm{\bm{x}}^k-\bm{x}^{k-1}\right\|^{\frac{1}{2\theta}}\left(\left(\mathcal{B}\right)^{\frac{1}{1-2\theta}}\left\|\bm{x}^k-\bm{x}^{k-1}\right\|\right)^{{\frac{1}{2\theta}}^{j-k+1}-{\frac{1}{2\theta}}}\nonumber\\
    & + \left(\mathcal{C}\right)^{\frac{1}{2\theta}}  \left\|\bm{e}^{k-1}\right\|^{\frac{1}{2\theta}}\left(\left(\mathcal{C}\right)^{\frac{1}{1-2\theta}}\left\|\bm{e}^{k-1}\right\|\right)^{{\frac{1}{2\theta}}^{j-k+1}-{\frac{1}{2\theta}}}.\label{eq:KL-case1-xbound}
\end{align}
Based on Lemma \ref{lem:relative-error}, Theorem \ref{theo:KL-global-convergence} and the fact that $\frac{1}{2\theta} > 1$, there must exist a small enough $\epsilon > 0$ and a large enough $k$ that following properties hold
\begin{align}
&\left\|\bm{e}^{k-1}\right\| = 0, \quad \left(\mathcal{B}\right)^{\frac{1}{1-2\theta}}\left\|\bm{x}^k-\bm{x}^{k-1}\right\|  \leq \epsilon,\nonumber\\ 
&\sum_{j=2}^{\infty} \epsilon^{{\frac{1}{2\theta}}^j-{\frac{1}{2\theta}}}=\epsilon^{{\frac{1}{2\theta}}^2-{\frac{1}{2\theta}}} \sum_{j=2}^{\infty} \epsilon^{{\frac{1}{2\theta}}^j-{\frac{1}{2\theta}}^2} \leq \frac{\xi}{2^{\frac{1+2\theta}{2\theta}} \left(\mathcal{B}\right)^{\frac{1}{2\theta}}}.
\label{eq:KL-case1-epslion-2}
\end{align}
Thus, inequality \eqref{eq:KL-case1-xbound} becomes
\begin{align}
    &\left\|\bm{x}^{j+1}-\bm{x}^j\right\|\leq \left(\mathcal{B}\right)^{\frac{1}{2\theta}}  \left\|\bm{\bm{x}}^k-\bm{x}^{k-1}\right\|^{\frac{1}{2\theta}}  \epsilon^{{\frac{1}{2\theta}}^{j-k+1}-{\frac{1}{2\theta}}}.
\label{eq:KL-case1-final-bound}
\end{align}
Summing over $j=k+1,k+2, \cdots, \infty$, we have
\begin{align}
    &\left\|\bm{x}^{k+1}-\bm{x}^{\star}\right\|\stackrel{(a)}\leq \sum_{j=k+1}^{\infty}\left\|\bm{x}^{j+1}-\bm{x}^j\right\|\nonumber\\
    & \stackrel{(b)}\leq \left(\mathcal{B}\right)^{\frac{1}{2\theta}}\left\|\bm{x}^{k+1}-\bm{x}^k\right\|^{\frac{1}{2\theta}} \sum_{j=2}^{\infty} \epsilon^{{\frac{1}{2\theta}}^{j}-{\frac{1}{2\theta}}}\nonumber\\
    & \stackrel{(c)}\leq  2^{-{\frac{1+2\theta}{2\theta}}} \xi\left\|\bm{x}^{k+1}-\bm{x}^k\right\|^{\frac{1}{2\theta}}\nonumber\\
    & \stackrel{(d)}\leq 0.5 \xi\left(\left\|\bm{x}^{k+1}-\bm{x}^{\star}\right\|^{\frac{1}{2\theta}}+\left\|\bm{x}^{k}-\bm{x}^{\star}\right\|^{\frac{1}{2\theta}}\right)\nonumber\leq \xi \left\|\bm{x}^{k}-\bm{x}^{\star}\right\|^{\frac{1}{2\theta}},\nonumber
\end{align}
where (a) comes from triangle inequality, (b) is due to \eqref{eq:KL-case1-final-bound}, (c) is due to \eqref{eq:KL-case1-epslion-2}, and (d) is due to the range of $\theta$ and triangle inequality.

\section{Proof of Proposition \ref{pro:rpca-theta}}
\label{Appendix-L}
To calculate the $\theta$ value of \eqref{eq:RPCA}, we first decouple variables $\bm{L}$ and $\bm{S}$ by replacing $\bm{L}$ with $\bm{Z}$ and adding an extra regularization term to force $\bm{L} \approx \bm{Z}$. Thus, we have
\begin{align}
\underset{\bm{L}, \bm{S}, \bm{Z} \in \mathbb{R}^{m \times n}}{\min} &\frac{1}{2}\left\|\bm{M}-\bm{Z}-\bm{S}\right\|_F^2+\delta_{\text {rank } \leq r}\left(\bm{L}\right) \nonumber \\
&+\lambda\left\|\bm{S}\right\|_0+ \frac{\alpha}{2}\left\|\bm{L}-\bm{Z}\right\|^{2}_{F},
\label{eq:RPCA-AppendixH-2}
\end{align}
where $\alpha \in (0,\infty)$. Equation \eqref{eq:RPCA-AppendixH-2} can be written as block separable sums of KL functions, i.e., $\sum_{\bm{X}_{i} \in (\bm{L}, \bm{S}, \bm{Z})} f_i\left(\bm{X}_{i}\right)$. By considering the Theorem 3.3 in \cite{li2018calculus} and the facts that $f(\bm{S})$ has the $\theta$ value of $\frac{1}{2}$ and $f(\bm{L})$ has the $\theta$ value of $1-\frac{1}{4.9^{\upsilon}}$ where $\upsilon=m n+m(m-r)+n(m-r)-1$ \cite[Section 5.3]{yu2022kurdyka}, the $\theta$ value of \eqref{eq:RPCA-AppendixH-2} is $\text{max}\{\frac{1}{2},1-\frac{1}{4.9^{\upsilon}}\}$. Since $\upsilon \gg 1$ in common practice, we conclude that the $\theta$ value is $1-\frac{1}{4.9^{\upsilon}}$ for the RPCA problem.

\printbibliography
\end{document}